\documentclass[a4paper,11pt,fleqn]{article}
\usepackage{amsmath,amssymb,amsthm,enumerate} 

\usepackage{hyperref}
\hypersetup{colorlinks, linkcolor=blue, citecolor=blue, urlcolor=blue}

\allowdisplaybreaks

\newcommand{\p}{\partial}
\newcommand{\ord}{\mathop{\rm ord}\nolimits}
\newcommand{\sgn}{\mathop{\rm sgn}\nolimits}

\newcommand{\todo}[1][\null]{\ensuremath{\clubsuit}}
\newcommand{\noprint}[1]{}

\newtheorem{theorem}{Theorem}
\newtheorem{lemma}[theorem]{Lemma}
\newtheorem{corollary}[theorem]{Corollary}
\newtheorem{proposition}[theorem]{Proposition}
{\theoremstyle{definition}

\newtheorem{remark}[theorem]{Remark}
}

\begin{document}

\noindent
{\LARGE\bf Extended symmetry analysis\\ of two-dimensional degenerate Burgers equation\par}

\vspace{4mm}

\noindent
Olena~O.~Vaneeva$^\dag$, Roman O. Popovych$^\ddag$ and Christodoulos Sophocleous$^\S$

\vspace{4mm}

\noindent{\it
\hbox to 2.8 mm{$^{\dag\ddag}$\hfil}Institute of Mathematics of NAS of Ukraine, 3 Tereshchenkivs'ka Str., 01024 Kyiv, Ukraine\\[1mm]
\hbox to 2.8 mm{$^\ddag$\hfil}Fakult\"at f\"ur Mathematik, Universit\"at Wien, Oskar-Morgenstern-Platz 1, A-1090 Wien, Austria\\[1mm]
\hbox to 2.8 mm{$^\S$\hfil}Department of Mathematics and Statistics, University of Cyprus, Nicosia CY 1678, Cyprus
}

\vspace{4mm}

{\noindent E-mails:  $^\dag$vaneeva@imath.kiev.ua, $^\ddag$rop@imath.kiev.ua, $^\S$christod@ucy.ac.cy}

\vspace{7mm}\par\noindent\hspace*{5mm}\parbox{150mm}{\small
We carry out the extended symmetry analysis of a two-dimensional degenerate Burgers equation.
Its complete point-symmetry group is found using the algebraic method,
and all its generalized symmetries are proved equivalent to its Lie symmetries.
We also prove that the space of conservation laws of this equation is infinite-dimensional
and is naturally isomorphic to the solution space of the (1+1)-dimensional backward linear heat equation.
Lie reductions of the two-dimensional degenerate Burgers equation are comprehensively studied in the optimal way
and new Lie invariant solutions are constructed.
We additionally consider solutions that also satisfy an analogous nondegenerate Burgers equation.
In total, we construct four families of solutions of two-dimensional degenerate Burgers equation
that are expressed in terms of arbitrary (nonzero) solutions of the (1+1)-dimensional linear heat equation.
Various kinds of hidden symmetries and hidden conservation laws
(local and potential ones) are discussed as well.
As a by-product, we exhaustively describe generalized symmetries, cosymmetries and conservation laws
of the transport equation, also called the inviscid Burgers equation,
and construct new invariant solutions of the nonlinear diffusion and diffusion--convection equations with power nonlinearities of degree~$-1/2$.
}\par\vspace{4mm}

\noprint{
\noindent{\footnotesize
\par}

}

\section{Introduction}

Differential equations modeling real-world phenomena and processes
were intensively studied within the framework of symmetry analysis of differential equations,
see, e.g., \cite{Ibragimov1994V1,blum2010a,BlumanBook1989,boch1999a,olve1993b,ovsi1982} and references therein.
Such studies may include, in particular, Lie symmetries, generalized, potential, conditional and hidden symmetries, cosymmetries,
local, potential and hidden conservation laws, Hamiltonian structures, reduction modules and recursion operators.
Symmetry methods also provide criteria for selecting significant models among parameterized families
of candidates~\cite{FN,ovsi1982,popo2010c,popo2017a}.
The above necessitates solving various classification problems for differential equations
and for objects associated with them within symmetry analysis.
At the same time, the investigation of such objects had been carried out at an adequate level
only for the simplest or the most important systems of differential equations.
Results presented in the literature for other important models are often incomplete or even incorrect 
\cite{kudr2009a,popo2010a}.

One of the latter models is the two-dimensional degenerate Burgers equation
\begin{gather}\label{eq:(1+2)DDegenerateBurgersEq}
u_t+uu_x-u_{yy}=0.
\end{gather}
The equation~\eqref{eq:(1+2)DDegenerateBurgersEq} has a number of applications.
In mathematical finance,
it arises in the course of simulating agents' decisions under risk via
representing agents' preferences over consumption processes by a refined utility functional
that takes into account the agents' habit formation~\cite{Citti,Lanconelli}.
The equation
\begin{gather}\label{eq:(1+2)DDegenerateBurgersEqEquivForm}
u_t= D u_{yy}-\nu\left(u(1-u)\right)_x,
\end{gather}
which is equivalent to the equation~\eqref{eq:(1+2)DDegenerateBurgersEq} under a simple point transformation,
can be related to the description of interacting particles of two kinds on a lattice~\cite{Alexander},
or more specifically, to the description of the dynamics of a `rod',
which is a large particle occupying several lattice sites,
as it moves in a fluid of monomers that each occupies only one site.
Here both the rod and the monomers are assumed to move by hopping to unoccupied neighboring sites,
interacting with each other through a hard-core exclusion, which prohibits two particles from occupying the same site.
It was shown in~\cite{Alexander} that certain features of such ensembles can be well described
by the equation~\eqref{eq:(1+2)DDegenerateBurgersEqEquivForm}. 

In contrast to the famous Cole--Hopf transformation 
(actually originally found by Forsyth in 1906, see \cite[Chapter~2, Notes]{olve1993b}) 
for the classical (1+1)-dimensional Burgers equation, 
no methods for linearizing the equation~\eqref{eq:(1+2)DDegenerateBurgersEq}, 
in particular using differential substitutions, are known.
Thus, consider the family of differential substitutions 
of the form $u=f(t,x,y,\tilde u,\tilde u_x,\tilde u_y)$, 
where the independent variables~$t$, $x$ and~$y$ are not changed,
which generalizes the Cole--Hopf transformation.
It can be proved by the direct computation that among such substitutions there are no substitution 
that maps the equation~\eqref{eq:(1+2)DDegenerateBurgersEqEquivForm} to 
a linear (1+2)-dimensional second-order evolution equation. 

The equations~\eqref{eq:(1+2)DDegenerateBurgersEq} and~\eqref{eq:(1+2)DDegenerateBurgersEqEquivForm}
are members of the more general class of nonlinear ultraparabolic equations
\begin{gather}\label{eq_2dDiff0}
u_t= D u_{yy}+\nu \left[K(u)\right]_x,
\end{gather}
where $D$ and $\nu$ are nonzero constants, $K(u)$ is a smooth nonlinear function of~$u$.
Equations of the form~\eqref{eq_2dDiff0} are called
nonlinear diffusion--advection equations or nonlinear Kolmogorov equations.
They describe diffusion--convection processes
with the directional separation of the diffusion and convection effects~\cite{Escobedo}
and arise in mathematical finance in the same context~\cite{Lanconelli,Pascucci}
as the more specific equation~\eqref{eq:(1+2)DDegenerateBurgersEq}.
Lie symmetries of these equations and their invariant solutions were considered in~\cite{deme2007a}.
The complete group classification of equations~\eqref{eq_2dDiff0}
with $D$ and $\nu$ being functions of time variable~$t$ was carried out in~\cite{vane2016a}.
After the simple substitution $u\mapsto -u$, the equation~\eqref{eq:(1+2)DDegenerateBurgersEq}
also becomes a member of the class of variable-coefficient (1+2)-dimensional Burgers equations
$u_t=uu_x+A(t)u_{xx}+B(t)u_{yy}$ treated from the Lie symmetry point of view in~\cite{ivan2010c}.
\looseness=-1

Some exact solutions of the equation~\eqref{eq:(1+2)DDegenerateBurgersEq} were constructed
using the Lie reduction method in a number of papers, see, e.g.,~\cite{deme2007a,elwa2006a,rass2018a,saie1999c}.
However, the listed families of solutions are not abundant.
Each of them can be obtained using a two-step Lie reduction,
where the first step is a particular codimension-one Lie reduction of the equation~\eqref{eq:(1+2)DDegenerateBurgersEq}
to the (1+1)-dimensional linear heat equation or to the Burgers equation.

In the present paper,
we carry out the enhanced classical Lie symmetry analysis
of the two-dimensional degenerate Burgers equation~\eqref{eq:(1+2)DDegenerateBurgersEq},
which includes the complete classification of Lie reductions,
but this is only a minor part of the consideration.
We also thoroughly study point symmetries, generalized symmetries, cosymmetries and local conservation laws
of the equation~\eqref{eq:(1+2)DDegenerateBurgersEq}.
We construct new families of exact solutions of~\eqref{eq:(1+2)DDegenerateBurgersEq},
which are much wider and sophisticated than those found in the literature.
In particular, four families of these solutions
are parameterized by the general solution of the (1+1)-dimensional linear heat equation.
As a by-product of the study of reduced equations,
new exact solutions for certain (1+1)-dimensional diffusion--convection equations are obtained.
We also present complete spaces of canonical representatives of generalized symmetries, cosymmetries and conservation laws
for the transport equation, also called the inviscid Burgers equation.

We begin the study of the equation~\eqref{eq:(1+2)DDegenerateBurgersEq}
in Section~\ref{sec:BurgersSystemMLIAlgAndCompleteGroup}
with constructing its complete point symmetry group~$G$,
which includes both continuous and discrete point symmetries.
In fact, we find independent discrete symmetries using
the automorphism-based version of the algebraic method suggested in~\cite{hydo2000b}
that involves factoring out inner automorphisms.
To derive the general form of point symmetries of~\eqref{eq:(1+2)DDegenerateBurgersEq},
we then compose discrete symmetries with the continuous ones,
which are generated by vector fields from the maximal Lie invariance algebra~$\mathfrak g$
of~\eqref{eq:(1+2)DDegenerateBurgersEq}.

To comprehensively carry out Lie reductions of~\eqref{eq:(1+2)DDegenerateBurgersEq}
to partial differential equations with two independent variables and to ordinary differential equations
in Sections~\ref{sec:BurgersSystemLiereductionsOfCodim1} and~\ref{sec:BurgersSystemLieReductionOfCodim2},
respectively, we first classify the one- and two-dimensional subalgebras of the algebra~$\mathfrak g$.
We use the optimized Lie reduction technique
that was developed in~\cite{fush1994a,fush1994b,popo1995b}
and adapted for multidimensional Burgers-like models in~\cite{kont2017a}.
The optimization is achieved due to a special selection 
of representatives within classes of equivalent subalgebras as well as of the form of related ansatzes.
For each of the listed inequivalent one-dimensional subalgebra of~$\mathfrak g$,
we construct a respective ansatz for~$u$,
derive the corresponding reduced partial differential equation,
and look for hidden symmetries of~\eqref{eq:(1+2)DDegenerateBurgersEq}
that are associated with this Lie reduction.
We distinguish four Lie reductions of codimension one leading to well-known equations,
which are two instances of the (1+1)-dimensional linear heat equation, the Burgers equation and the transport equation.
Only these four reductions among the listed inequivalent Lie reductions
result in nontrivial hidden symmetries of~\eqref{eq:(1+2)DDegenerateBurgersEq}, both Lie and generalized ones.
Taking into account the Cole--Hopf transformation,
we get three families of invariant solutions
that are parameterized by the general solution of the (1+1)-dimensional linear heat equation.
The only essential Lie reductions of codimension two for~\eqref{eq:(1+2)DDegenerateBurgersEq} are those
related to two-dimensional subalgebras of~$\mathfrak g$ that do not contain, up to $G$-equivalence,
the one-dimensional subalgebras associated with the four distinguished Lie reductions of codimension one.
Integrating reduced ordinary differential equations
obtained in the course of inequivalent essential Lie reductions of codimension two,
we construct parameterized families of new invariant solutions of the equation~\eqref{eq:(1+2)DDegenerateBurgersEq},
which substantially differ from the known ones.

In Section~\ref{sec:(1+2)DDegenerateBurgersEqGenSyms}, we prove that
the quotient algebra of generalized symmetries of~\eqref{eq:(1+2)DDegenerateBurgersEq} 
with respect to the subalgebra of trivial symmetries 
is naturally isomorphic to its maximal Lie invariance algebra~$\mathfrak g$.
The result was predictable but its proof is sophisticated and cumbersome.

Section~\ref{sec:(1+2)DDegenerateBurgersEqCLs} concerns cosymmetries and conservation laws
of the equation~\eqref{eq:(1+2)DDegenerateBurgersEq}.
Any cosymmetry coset of~\eqref{eq:(1+2)DDegenerateBurgersEq} contains a cosymmetry of order~$-\infty$,
which is hence necessarily a conservation-law characteristic of~\eqref{eq:(1+2)DDegenerateBurgersEq}.
Therefore, both the analogous quotient spaces of cosymmetries and of conservation-law characteristics
of~\eqref{eq:(1+2)DDegenerateBurgersEq} are naturally isomorphic to the solution space
of the (1+1)-dimensional backward heat equation $\gamma_t+\gamma_{yy}=0$,
and the corresponding space of conserved currents canonically representing conservation laws
of~\eqref{eq:(1+2)DDegenerateBurgersEq} can be easily constructed.
Conservations laws of all inequivalent reduced partial differential equations except one,
which is the transport equation, are induced by conservation laws of the original equation~\eqref{eq:(1+2)DDegenerateBurgersEq}.
We compute the spaces of cosymmetries, conservation laws and conservation-law characteristics 
of the transport equation since we were not able to find their description in the literature.
We analyze which conservation laws of the transport equation can be interpreted as
hidden conservation laws of the equation~\eqref{eq:(1+2)DDegenerateBurgersEq}.
One more family of hidden conservation laws of~\eqref{eq:(1+2)DDegenerateBurgersEq}
is given by the potential conservation laws of the reduced equation coinciding with the Burgers equation.

In Section~\ref{sec:CommonSolutionsOf(1+2)DDegenerateAndNondegenerateBurgersEqs},
we consider solutions that are common for the equation~\eqref{eq:(1+2)DDegenerateBurgersEq}
and the (1+2)-dimensional nondegenerate Burgers equation $u_t+uu_x-u_{xx}-u_{yy}=0$.
Such solutions are affine in~$x$,
and we review wide families of them that were constructed in an explicit form in~\cite{kont2017a,raja2008a}.
In particular, the set of such solutions contains,
in addition to two of the above three similar families of invariant solutions,
one more family of solutions
that is parameterized by the general solution of the (1+1)-dimensional linear heat equation.

An enhanced complete list of inequivalent known closed-form solutions
of (1+1)-dimensional linear heat equation is given in Section~\ref{sec:ExactSolutionsOfHeatEq}.
Finally, in Section~\ref{sec:ExactSolutionsOfNDCEs}
we present new closed-form invariant solutions for
the nonlinear diffusion and diffusion--convection equations with power nonlinearities of degree~$-1/2$,
$v_t=(v^{-1/2}v_x)_x$ and \mbox{$v_t=(v^{-1/2}v_x)_x+v^{-1/2}v_x$}.
These equations had been intensively studied within the framework of symmetry analysis of differential equations
and are mapped by a simple substitution to reduced equations of~\eqref{eq:(1+2)DDegenerateBurgersEq}.

One more kind of symmetry-like objects that could be studied
for the (1+2)-dimensional degenerate Burgers equation~\eqref{eq:(1+2)DDegenerateBurgersEq}
is given by reduction modules~\cite{boyk2016a}, 
which are associated with nonclassical reductions introduced in~\cite{blum1969a}.
However, the classification of reduction modules
for the equation~\eqref{eq:(1+2)DDegenerateBurgersEq} is an interesting but complicated problem,
which deserves a separate consideration.

\section{Lie invariance algebra and complete point-symmetry group}\label{sec:BurgersSystemMLIAlgAndCompleteGroup}

The classical approach for deriving Lie symmetries is well known and was established in the last decades \cite{BlumanBook1989,olve1993b,ovsi1982}.
The maximal Lie invariance algebra of the two-dimensional degenerate Burgers equation~\eqref{eq:(1+2)DDegenerateBurgersEq} is
\[
\mathfrak g=\langle D^t,\,D^x,\,P^t,\,G^x,\,P^y,\,P^x\rangle,
\]
where
$D^t=2t\p_t+y\p_y-2u\p_u$,
$D^x=x\p_x+u\p_u$,
$P^t=\p_t$,
$G^x=t\p_x+\p_u$,
$P^y=\p_y$,
$P^x=\p_x$.
The nonzero commutation relations of these vector fields, up to antisymmetry, are the following
\begin{gather*}
[P^t,D^t]= 2P^t,\quad
[G^x,D^t]=-2G^x,\quad
[P^y,D^t]=  P^y,\\
[G^x,D^x]=  G^x,\quad
[P^x,D^x]=  P^x,\quad
[P^t,G^x]=  P^x.
\end{gather*}
The algebra~$\mathfrak g$ is solvable since
$\mathfrak g'=\langle P^t,\,G^x,\,P^y,\,P^x\rangle$,
$\mathfrak g''=\langle P^x\rangle$ and $\mathfrak g'''=\{0\}$.
The nilradical~$\mathfrak n$ of~$\mathfrak g$ coincides with~$\mathfrak g'$.
This is why the natural ranking of the basis elements of~$\mathfrak g$ is
\begin{gather}\label{eq:RankingOfBasisElementsOfMIA}
D^t\succ D^x\succ P^t\succ G^x\succ P^y\succ P^x.
\end{gather}
We fix the basis $(D^t,D^x,P^t,G^x,P^y,P^x)$ of~$\mathfrak g$.
The basis elements generate
simultaneous scalings of $t$, $y$ and~$u$,
scalings of $(x,u)$,
shifts in~$t$,
Galilean boosts in~$x$,
shifts in~$y$ and
shifts in~$x$, respectively.

The automorphism group $\mathrm{Aut}(\mathfrak g)$ of~$\mathfrak g$ consists of the linear operators on~$\mathfrak g$
whose matrices are, in the fixed basis, of the form
\[
\left(
\begin{array}{cccccc}
1&0&0&0&0&0\\
0&1&0&0&0&0\\
a_{31}&0&a_{33}&0&0&0\\
-2 a_{42}&a_{42}&0&a_{44}&0&0\\
a_{51}&0&0&0&a_{55}&0\\
-a_{31}a_{42}&a_{62}&-a_{33}a_{42}&\frac12a_{31}a_{44}&0&a_{33}a_{44}
\end{array}
\right),
\]
where the parameters~$a$'s are arbitrary constants with $a_{33}a_{44}a_{55}\ne0$.
The complete list of essential proper megaideals%
\footnote{%
We recall that a \emph{fully characteristic ideal} \cite[Exercise~14.1.1]{hilg2011a} (or, shortly, \emph{megaideal} \cite{bihl2011a,popo2003a})
of a Lie algebra is a subspace of this algebra that is invariant with respect to the automorphism group of this algebra.
}
of the algebra~$\mathfrak g$,
which are not sums of other megaideals, is exhausted by the following spans:
\begin{gather*}
\mathfrak m_1=\langle P^x\rangle,\quad
\mathfrak m_2=\langle P^y\rangle,\quad
\mathfrak m_3=\langle G^x,P^x\rangle,\quad
\mathfrak m_4=\langle P^t,P^x\rangle,\quad
\mathfrak m_5=\langle D^x,G^x,P^x\rangle,\\
\mathfrak m_6=\langle D^t+2D^x,P^t,P^y,P^x\rangle,\quad
\mathfrak m_{7,\alpha}=\langle D^t+\alpha D^x,P^t,G^x,P^y,P^x\rangle,\ \alpha\ne2.
\end{gather*}
The other proper megaideals are
\begin{gather*}
\mathfrak m_1+\mathfrak m_2=\langle P^y,P^x\rangle,\quad
\mathfrak m_2+\mathfrak m_3=\langle G^x,P^y,P^x\rangle,\quad
\mathfrak m_2+\mathfrak m_4=\langle P^t,P^y,P^x\rangle,\\
\mathfrak m_3+\mathfrak m_4=\langle P^t,G^x,P^x\rangle,\quad
\mathfrak m_2+\mathfrak m_3+\mathfrak m_4=\langle P^t,G^x,P^y,P^x\rangle,\\
\mathfrak m_2+\mathfrak m_5=\langle D^x,G^x,P^y,P^x\rangle,\quad
\mathfrak m_4+\mathfrak m_5=\langle D^x,P^t,G^x,P^x\rangle,\\
\mathfrak m_2+\mathfrak m_4+\mathfrak m_5=\langle D^x,P^t,G^x,P^y,P^x\rangle,\quad
\mathfrak m_3+\mathfrak m_6=\langle D^t+2D^x,P^t,G^x,P^y,P^x\rangle=\mathfrak m_{7,2}.\quad
\end{gather*}

Only some megaideals of the algebra~$\mathfrak g$ can be related to its structural objects
or, more widely, can be recursively computed
starting from the improper megaideals~$\{0\}$ and $\mathfrak g$ and using various techniques
without explicitly involving the group $\mathrm{Aut}(\mathfrak g)$;
see, e.g., a collection of these techniques in~\cite{bihl2015a}
and their applications in~\cite{bihl2011a,card2012a,popo2003a}.
Thus,
\[
\mathfrak m_2+\mathfrak m_3+\mathfrak m_4=\mathfrak g'=\mathfrak n,\quad
\mathfrak m_1=\mathfrak g'',\quad
\mathfrak m_1+\mathfrak m_2=\mathrm Z_{\mathfrak g'},
\]
where $\mathrm Z_{\mathfrak a}$ denotes the center of a Lie algebra~$\mathfrak a$.
To present a structural interpretation of other megaideals of~$\mathfrak g$,
we apply Proposition~1 of~\cite{card2012a}.
It states that
if~$\mathfrak i_0$, $\mathfrak i_1$ and~$\mathfrak i_2$ are megaideals of~$\mathfrak g$,
then the set~$\mathfrak s$ of elements from~$\mathfrak i_0$
whose commutators with arbitrary elements from~$\mathfrak i_1$ belong to~$\mathfrak i_2$
is also a megaideal of~$\mathfrak g$.
We first choose $(\mathfrak i_0,\mathfrak i_1,\mathfrak i_2)=(\mathfrak g,\mathrm Z_{\mathfrak g'},\mathfrak g'')$,
which gives $\mathfrak s_1=\mathfrak m_2+\mathfrak m_4+\mathfrak m_5$.
Then $\mathfrak m_2=\mathrm Z_{\mathfrak s_1}$, $\mathfrak m_3=\mathfrak s_1'$
and thus $\mathfrak m_2+\mathfrak m_3=\mathrm Z_{\mathfrak s_1}+\mathfrak s_1'$.
For $(\mathfrak i_0,\mathfrak i_1,\mathfrak i_2)=(\mathfrak g,\mathfrak s_1,\mathfrak s_1')$
and $(\mathfrak i_0,\mathfrak i_1,\mathfrak i_2)=(\mathfrak s_1,\mathfrak s_1,\mathfrak g'')$
we respectively obtain $\mathfrak s_2=\mathfrak m_5$ and $\mathfrak s_3=\mathfrak m_2+\mathfrak m_4$.
On the next iteration we take
$(\mathfrak i_0,\mathfrak i_1,\mathfrak i_2)=(\mathfrak g,\mathfrak g,\mathfrak s_3)$
to obtain $\mathfrak s_4=\mathfrak m_6$.
Hence, $\mathfrak m_3+\mathfrak m_6=\mathfrak s_1'+\mathfrak s_4$.

The megaideals~$\mathfrak m_4$, $\mathfrak m_3+\mathfrak m_4$, $\mathfrak m_4+\mathfrak m_5$
and $\mathfrak m_{7,\alpha}$ with $\alpha\ne2$
do not admit an interpretation of the above kind.
Their occurrence is explained by the fact
that only some constraints on automorphisms of~$\mathfrak g$
can be described in terms of commutators of subspaces of~$\mathfrak g$.

\begin{lemma}\label{lem:(1+2)DDegenerateBurgersEqDiscreteSyms}
A complete list of discrete symmetry transformations
of the (1+2)-dimensional degenerate Burgers equation~\eqref{eq:(1+2)DDegenerateBurgersEq}
that are independent up to combining with each other and with continuous symmetry transformations of this equation
is exhausted by two transformations alternating signs of variables,
$(t,x,y,u)\mapsto(t,-x,y,-u)$ and $(t,x,y,u)\mapsto(t,x,-y,u)$.
\end{lemma}

\begin{proof}
The maximal Lie invariance algebra~$\mathfrak g$ of the equation~\eqref{eq:(1+2)DDegenerateBurgersEq}
is nontrivial and finite-dimensional.
The automorphism group~$\mathrm{Aut}(\mathfrak g)$ is easily computed.
It is not much wider than the inner automorphism group~$\mathrm{Inn}(\mathfrak g)$ of~$\mathfrak g$,
which is constituted by the linear operators on~$\mathfrak g$ with matrices of the form
\[
\left(
\begin{array}{cccccc}
1&0&0&0&0&0\\
0&1&0&0&0&0\\
2{\rm e}^{2\delta_1}\delta_3&0&{\rm e}^{2\delta_1}&0&0&0\\
2{\rm e}^{\delta_2-2\delta_1}\delta_4&-{\rm e}^{\delta_2-2\delta_1}\delta_4&0&{\rm e}^{\delta_2-2\delta_1}&0&0\\
{\rm e}^{\delta_1}\delta_5&0&0&0&{\rm e}^{\delta_1}&0\\
2{\rm e}^{\delta_2}\delta_3\delta_4&{\rm e}^{\delta_2}(\delta_6-\delta_3\delta_4)&{\rm e}^{\delta_2}\delta_4&{\rm e}^{\delta_2}\delta_3&0&{\rm e}^{\delta_2}
\end{array}
\right),
\]
where the parameters~$\delta_1$, \dots, $\delta_6$ are arbitrary constants.
This is why we apply the automor\-phism-based version of the algebraic method
of finding the complete point-symmetry groups of differential equations
that involves factoring out inner automorphisms.
This version was suggested in~\cite{hydo2000b} and further developed in~\cite{bihl2015a,card2012a,kont2017a}.
Then, we use constraints obtained by the algebraic method for components of point symmetry transformations
to complete the proof with the direct method.
See, e.g., \cite{king1998a} for techniques of computing point transformations between differential equations by the direct method.

The quotient group $\mathrm{Aut}(\mathfrak g)/\mathrm{Inn}(\mathfrak g)$
can be identified with the matrix group consisting of the diagonal matrices of the form
$\mathrm{diag}(1,1,b,\varepsilon/b,\varepsilon',\varepsilon)$,
where $\varepsilon,\varepsilon'=\pm1$ and
the parameter~$b$ runs through~$\mathbb R\setminus\{0\}$.
Suppose that the pushforward~$\mathcal T_*$ of vector fields in the space with coordinates $(t,x,y,u)$ by a point transformation
\[
\mathcal T\colon\quad (\tilde t,\tilde x,\tilde y,\tilde u)=(T,X,Y,U)(t,x,y,u)
\]
is the automorphism of~$\mathfrak g$ with the matrix $\mathrm{diag}(1,1,b,\varepsilon/b,\varepsilon',\varepsilon)$, i.e.,
\begin{gather*}
\mathcal T_*P^x=\varepsilon \tilde P^x,\ \
\mathcal T_*P^y=\varepsilon'\tilde P^y,\ \
\mathcal T_*G^x=\varepsilon b^{-1}\tilde G^x,\ \
\mathcal T_*P^t=b\tilde P^t,\ \
\mathcal T_*D^x=\tilde D^x,\ \
\mathcal T_*D^t=\tilde D^t,
\end{gather*}
where tildes over vector fields mean that these vector fields are given in the new coordinates.
We componentwise split the above conditions for~$\mathcal T_*$
and thus derive a system of differential equations for the components of~$\mathcal T$,
\begin{gather*}
X_x=\varepsilon,\quad T_x=Y_x=U_x=0,\\
Y_y=\varepsilon',\quad T_y=X_y=U_y=0,\\
tX_x+X_u=\varepsilon b^{-1}T,\quad U_u=\varepsilon b^{-1},\quad T_u=Y_u=0,\\
T_t=b,\quad X_t=Y_t=U_t=0,\\
xX_x+uX_u=X,\quad uU_u=U,\quad tT_t=T,\quad yY_y=Y.
\end{gather*}
This system implies that $T=bt$ and hence $X_u=0$.
Furthermore, $X=\varepsilon x$, $Y=\varepsilon'y$ and $U=\varepsilon b^{-1}u$.

Using the chain rule, we express all the required transformed derivatives in terms of the initial coordinates,
$\tilde u_{\tilde t}=\varepsilon b^{-2}u_t$,
$\tilde u_{\tilde x}=b^{-1}u_x$ and
$\tilde u_{\tilde y\tilde y}=\varepsilon b^{-1}u_x$.
Then we substitute the obtained expressions into the copy of the equation~\eqref{eq:(1+2)DDegenerateBurgersEq}
in the new coordinates.
The expanded equation should identically be satisfied by each solution of the equation~\eqref{eq:(1+2)DDegenerateBurgersEq}.
This condition implies that $b=1$.
Therefore, discrete symmetries of the equation~\eqref{eq:(1+2)DDegenerateBurgersEq} are exhausted,
up to combining with continuous symmetries and with each other,
by the two involutions $(t,x,y,u)\mapsto(t,-x,y,-u)$ and $(t,x,y,u)\mapsto(t,x,-y,u)$,
which are associated with the values $(\varepsilon,\varepsilon')=(-1,1)$ and \mbox{$(\varepsilon,\varepsilon')=(1,-1)$},
respectively.
\end{proof}

\begin{corollary}
The factor group of the complete point-symmetry group~$G$
of the (1+2)-dimen\-sional degenerate Burgers equation~\eqref{eq:(1+2)DDegenerateBurgersEq}
with respect to its identity component is isomorphic to the group $\mathbb Z_2\times\mathbb Z_2$.
\end{corollary}

Therefore, the complete point-symmetry group~$G$ of the (1+2)-dimensional degenerate Burgers equation~\eqref{eq:(1+2)DDegenerateBurgersEq}
is generated by one-parameter point-transformation groups associated with vector fields from the algebra~$\mathfrak g$
and the two discrete transformations given in Lemma~\ref{lem:(1+2)DDegenerateBurgersEqDiscreteSyms}.

\begin{theorem}\label{thm:BurgersSystemCompletePointSymGroup}
The complete point-symmetry group~$G$ of the (1+2)-dimensional degenerate Burgers equation~\eqref{eq:(1+2)DDegenerateBurgersEq}
consists of the point transformations
\begin{gather*}
\tilde t=\delta_1^2t+\delta_3, \quad
\tilde x=\delta_2x+\delta_4t+\delta_6, \quad
\tilde y=\delta_1y+\delta_5, \quad
\tilde u=\delta_1^{-2}\delta_2u+\delta_4,
\end{gather*}
where the parameters~$\delta_1$, \dots, $\delta_6$ are arbitrary constants with $\delta_1\delta_2\ne0$.
\end{theorem}

\section{Lie reductions of codimension one}\label{sec:BurgersSystemLiereductionsOfCodim1}

We classify one-dimensional subalgebras of the algebra~$\mathfrak g$
up to the equivalence relation generated by the induced
adjoint action of the point-symmetry group~$G$ of the equation~\eqref{eq:(1+2)DDegenerateBurgersEq} on~$\mathfrak g$.
Since the Lie algebra~$\mathfrak g$ is solvable,
we explore the classification cases in an order that is consistent with the ranking~\eqref{eq:RankingOfBasisElementsOfMIA}.
The selected one-dimensional $G$-inequivalent subalgebras are additionally arranged
in order to be convenient for successive Lie reductions.

A complete list of one-dimensional $G$-inequivalent subalgebras of~$\mathfrak g$ is constituted by the subalgebras
\begin{gather*}
\mathfrak g_{1.1}^\kappa             =\langle D^t+2\kappa D^x                  \rangle,\quad
\mathfrak g_{1.2}                    =\langle D^t+2P^x                         \rangle,\quad
\mathfrak g_{1.3}                    =\langle D^t+2D^x+2G^x                    \rangle,\\
\mathfrak g_{1.4}^{\varepsilon'\beta}=\langle D^x-\varepsilon'P^t+\beta P^y    \rangle_{\varepsilon'=\pm1,\,\beta\geqslant0},\quad
\mathfrak g_{1.5}^\delta             =\langle D^x-\delta P^y                   \rangle_{\delta\in\{0,1\}},\\
\mathfrak g_{1.6}^{\delta\,\delta'}  =\langle P^t+\delta G^x+\delta'P^y        \rangle_{\delta,\delta'\in\{0,1\}},\quad
\mathfrak g_{1.7}                    =\langle G^x+P^y                          \rangle,\\
\mathfrak g_{1.8}                    =\langle G^x                              \rangle,\quad
\mathfrak g_{1.9}                    =\langle P^x                              \rangle,\quad
\mathfrak g_{1.10}                   =\langle P^y-P^x                          \rangle,\quad
\mathfrak g_{1.11}                   =\langle P^y                              \rangle.
\end{gather*}

Ansatzes constructed with one-dimensional subalgebras of~$\mathfrak g$
reduce the equation~\eqref{eq:(1+2)DDegenerateBurgersEq}
to partial differential equations in two independent variables.
In Table~\ref{tab:LieReductionsOfCodim1}, for each of the listed one-dimensional subalgebras,
we present an ansatz constructed for $u$ and the corresponding reduced equation.

\begin{table}[!ht]\footnotesize
\caption{Lie reductions of codimension one}
\label{tab:LieReductionsOfCodim1}
\begin{center}
\renewcommand{\arraystretch}{1.4}
\begin{tabular}{|r|c|c|c|l|}
\hline
no.& $u$                                 & $z_1$                  & $z_2$                   & \hfil Reduced equation\\
\hline
 1.1 & $t^{-1}|t|^\kappa w+\kappa t^{-1}x$ & $|t|^{-\kappa}x$       & $|t|^{-1/2}y$           & $ww_1=\varepsilon w_{22}+\tfrac12z_2w_2-(2\kappa-1)w-\kappa(\kappa-1)z_1$ \\
 1.2 & $t^{-1}w+t^{-1}$                    & $x-\ln|t|$             & $|t|^{-1/2}y$           & $ww_1=\varepsilon w_{22}+\tfrac12z_2w_2+w+1$ \\
 1.3 & $w+t^{-1}x+1$                       & $t^{-1}x-\ln|t|$       & $|t|^{-1/2}y$           & $ww_1=\varepsilon w_{22}+\tfrac12z_2w_2-w-1$ \\
 1.4 & $xw+\varepsilon'x$                  & $\ln|x|-\varepsilon't$ & $\varepsilon'y-\beta t$ & $ww_1=w_{22}+\beta w_2-(w+\varepsilon')^2$ \\
 1.5 & $xw$                                & $t$                    & $y+\delta\ln|x|$        & $w_1+\delta ww_2=w_{22}-w^2$ \\
 1.6 & $w+\delta t$                        & $x-\frac12\delta t^2$  & $y-\delta't$            & $ww_1=w_{22}+\delta'w_2-\delta$ \\
 1.7 & $w+y$                               & $t$                    & $x-ty$                  & $w_1+ww_2=z_1^2w_{22}$ \\
 1.8 & $t^{-1}w+t^{-1}x$                   & $t$                    & $y$                     & $w_1=w_{22}$ \\
 1.9 & $w$                                 & $t$                    & $y$                     & $w_1=w_{22}$ \\
1.10 & $w$                                 & $t$                    & $x+y$                   & $w_1+ww_2=w_{22}$ \\
1.11 & $w$                                 & $t$                    & $x$                     & $w_1+ww_2=0$ \\
\hline
\end{tabular}
\end{center}
Here $\varepsilon=\sgn t$, $\varepsilon'=\pm1$, $\beta\geqslant0$, $\delta,\delta'\in\{0,1\}$,
$w=w(z_1,z_2)$ is the new unknown function of the invariant independent variables $(z_1,z_2)$.
The subscripts~1 and~2 of~$w$ denote derivatives with respect to~$z_1$ and~$z_2$, respectively.
\end{table}


The maximal Lie invariance algebras of the reduced equations presented in Table~1 are respectively
\begin{gather*}
\mathfrak a_1=\langle\tilde D^1\rangle \quad\mbox{if}\quad \kappa\notin\{0,1\} \qquad\mbox{and}\qquad
\mathfrak a_1=\langle\tilde P^1,\ \tilde D^1\rangle \quad\mbox{if}\quad \kappa\in\{0,1\};
\\
\mathfrak a_2=\langle\tilde P^1 \rangle;
\qquad
\mathfrak a_3=\langle\tilde P^1\rangle; \qquad \mathfrak a_4=\langle\tilde P^1,\ \tilde P^2\rangle;
\\
\mathfrak a_5=\langle\tilde P^1,\,\tilde P^2\rangle \quad\mbox{if}\quad \delta=1 \qquad\mbox{and}\qquad
\mathfrak a_5=\langle\tilde P^1,\,\tilde P^2,\,2z_1\p_{z_1}+z_2\p_{z_2}-2w\p_w\rangle \quad\mbox{if}\quad \delta=0;
\\
\mathfrak a_6=\langle\tilde P^1,\,\tilde P^2\rangle \quad\mbox{if}\quad \delta=\delta'=1,\qquad
\mathfrak a_6=\langle\tilde P^1,\,\tilde P^2,\,4\tilde D^1+\tilde D^2\rangle \quad\mbox{if}\quad \delta=1,\ \delta'=0,\\
\mathfrak a_6=\langle\tilde P^1,\,\tilde P^2,\,\tilde D^1\rangle \quad\mbox{if}\quad \delta=0,\ \delta'=1 \qquad\mbox{and}\qquad
\mathfrak a_6=\langle\tilde P^1,\,\tilde P^2,\,\tilde D^1,\,\tilde D^2\rangle \quad\mbox{if}\quad \delta=\delta'=0;\\
\mathfrak a_7=\langle\p_{z_2},\ z_1\p_{z_2}-\p_w,\ 2z_1\p_{z_1}+3z_2\p_{z_2}+w\p_w\rangle;\\
\mathfrak a_8=\mathfrak a_9=\langle\p_{z_1},\ \p_{z_2},\ 2z_1\p_{z_2}-z_2w\p_w,\
2z_1\p_{z_1}+z_2\p_{z_2},\ 4z_1^2\p_{z_1}+4z_1z_2\p_{z_2}-(z_2^2+2z_1)w\p_w,\\ \phantom{\mathfrak a_8=\mathfrak a_9=\langle}
w\p_w,\ f(z_1,z_2)\p_w\rangle;\\
\mathfrak a_{10}=\langle\p_{z_1},\ \p_{z_2},\ z_1\p_{z_2}-\p_w,\ 2z_1\p_{z_1}+z_2\p_{z_2}-w\p_w,\
z_1^2\p_{z_1}+z_1z_2\p_{z_2}-(z_1w+z_2)\p_w\rangle;\\
\mathfrak a_{11}=\langle\tau(z_1,z_2,w)(\p_{z_1}+w\p_{z_2}),\ \xi(z_2-wz_1,w)\p_{z_2},\ \eta(z_2-wz_1,w)(z_1\p_{z_2}+\p_w)\rangle.
\end{gather*}
Here the function $f=f(z_1,z_2)$ runs through the solution set of the heat equation $f_{z_1}=f_{z_2z_2}$,
$\tau$, $\xi$ and~$\eta$ are arbitrary smooth functions of their arguments,
and we denote
\begin{gather*}
\tilde P^1=\p_{z_1},\quad
\tilde P^2=\p_{z_2},\quad
\tilde D^1=z_1\p_{z_1}+w\p_w,\quad
\tilde D^2=z_2\p_{z_2}-2w\p_w.
\end{gather*}
The above results have been checked using the package \texttt{DESOLV}
for symbolic calculations of Lie symmetries~\cite{carm2000a,vu2007a}.

The (1+2)-dimensional degenerate Burgers equation~\eqref{eq:(1+2)DDegenerateBurgersEq}
admits \emph{additional}~\cite{olve1993b} or \emph{hidden}~\cite{abra2008a} symmetries
with respect to a Lie reduction if the corresponding reduced equation possesses Lie symmetries
that are not induced by Lie symmetries of the original equation~\eqref{eq:(1+2)DDegenerateBurgersEq}.
Note that the first example of such symmetries was constructed in~\cite{kapi1978a}
for the axisymmetric reduction of the incompressible Euler equations,
see also the discussion of this example in~\cite[Example~3.5]{olve1993b}.
The comprehensive study of hidden symmetries of an important model was first carried out
for the Navier--Stokes equations in~\cite{fush1994a,fush1994b}.
To check which Lie symmetries of reduced equations~1.1--1.11 are induced
by Lie symmetries of the original equation~\eqref{eq:(1+2)DDegenerateBurgersEq},
for each $m\in\{1,\dots,11\}$ we compute the normalizer of the subalgebra~$\mathfrak g_{1.m}$
in the algebra~$\mathfrak g$,
\[{\rm N}_{\mathfrak g}(\mathfrak g_{1.m})=\{Q\in\mathfrak g\mid [Q,Q']\in\mathfrak g_{1.m}\ \mbox{for all}\ Q'\in\mathfrak g_{1.m}\}.\]
The algebra of induced Lie symmetries of reduced equation~1.$m$
is isomorphic to the quotient algebra ${\rm N}_{\mathfrak g}(\mathfrak g_{1.m})/\mathfrak g_{1.m}$.
Therefore, all Lie symmetries of reduced equation~1.$m$
are induced by Lie symmetries of the original Burgers equation~\eqref{eq:(1+2)DDegenerateBurgersEq}
if and only if  $\dim\mathfrak a_m=\dim{\rm N}_{\mathfrak g}(\mathfrak g_{1.m})-1$.
Thus, respectively
\begin{gather*}
{\rm N}_{\mathfrak g}(\mathfrak g_{1.1}^\kappa)=\langle D^t,D^x\rangle,\,\langle D^t,D^x,P^x\rangle,\,\langle D^t,D^x,G^x\rangle
\quad\mbox{if}\quad \kappa\notin\{0,1\},\ \kappa=0,\ \kappa=1,
\\
{\rm N}_{\mathfrak g}(\mathfrak g_{1.2})=\langle D^t,P^x\rangle,\quad
{\rm N}_{\mathfrak g}(\mathfrak g_{1.3})=\langle D^t+2D^x,G^x\rangle,\quad
\smash{{\rm N}_{\mathfrak g}(\mathfrak g_{1.4}^{\varepsilon'\beta})}=\langle D^x,P^t,P^y\rangle,
\\
{\rm N}_{\mathfrak g}(\mathfrak g_{1.5}^\delta)=\langle D^x,P^t,P^y\rangle,\,\langle D^t,D^x,P^t,P^y\rangle
\quad\mbox{if}\quad \delta=1,\ \delta=0,
\\
\smash{{\rm N}_{\mathfrak g}(\mathfrak g_{1.6}^{\delta\,\delta'})}=
\langle P^t+\delta G^x,P^y,P^x\rangle,\,\langle D^x,P^t+\delta G^x,P^y,P^x\rangle,\,
\langle D^t+4D^x,P^t+\delta G^x,P^y,P^x\rangle,\\ \quad \langle D^t,D^x,P^t+\delta G^x,P^y,P^x\rangle
\quad\mbox{if}\quad\delta=\delta'=1,\ (\delta,\delta')=(0,1),\ (\delta,\delta')=(1,0),\ \delta=\delta'=0,
\\
{\rm N}_{\mathfrak g}(\mathfrak g_{1.7})=\langle D^t+3D^x,G^x,P^y,P^x\rangle,\quad
{\rm N}_{\mathfrak g}(\mathfrak g_{1.8})=\langle D^t,D^x,G^x,P^y,P^x\rangle,
\\
{\rm N}_{\mathfrak g}(\mathfrak g_{1.9})=\mathfrak g,\quad
{\rm N}_{\mathfrak g}(\mathfrak g_{1.10})=\langle D^t+D^x,P^t,G^x,P^y,P^x\rangle,\quad
{\rm N}_{\mathfrak g}(\mathfrak g_{1.11})=\mathfrak g.
\end{gather*}

Comparing the dimensions of ${\rm N}_{\mathfrak g}(\mathfrak g_{1.m})$ and~$\mathfrak a_m$,
we conclude that
all Lie symmetries of reduced equations~1.1--1.7 are induced by Lie symmetries
of the original Burgers equation~\eqref{eq:(1+2)DDegenerateBurgersEq},
but this is not the case for reduced equations~1.8--1.11.
The subalgebras of induced symmetries in the algebras~$\mathfrak a_8$--$\mathfrak a_{11}$ are respectively
\begin{gather*}
\tilde{\mathfrak a}_8=\langle\p_{z_1},\ 2z_1\p_{z_1}+z_2\p_{z_2},\ w\p_w,\ \p_w\rangle,\quad
\tilde{\mathfrak a}_9=\langle\p_{z_1},\ \p_{z_2},\ 2z_1\p_{z_1}+z_2\p_{z_2},\ w\p_w,\ \p_w\rangle,\\
\tilde{\mathfrak a}_{10}=\langle\p_{z_1},\ \p_{z_2},\ 2z_1\p_{z_1}+z_2\p_{z_2}-2w\p_w,\ z_1\p_{z_2}+\p_w\rangle,\\
\tilde{\mathfrak a}_{11}=\langle\p_{z_1},\ \p_{z_2},\ 2z_1\p_{z_1}+w\p_w,\ z_2\p_{z_2}+w\p_w,\ z_1\p_{z_2}+\p_w\rangle.
\end{gather*}
For each $i\in\{8,\dots,11\}$, the elements of the complement $\mathfrak a_i\setminus\tilde{\mathfrak a}_i$
of $\tilde{\mathfrak a}_i$ in~$\mathfrak a_i$ are nontrivial hidden symmetries of the equation~\eqref{eq:(1+2)DDegenerateBurgersEq}
associated with reduction~1.$i$.

Therefore, the study of further Lie reductions of reduced equations~1.1--1.7 to ordinary differential equations is needless
since it is more efficient to directly reduce the equation~\eqref{eq:(1+2)DDegenerateBurgersEq} to ordinary differential equations
using two-dimensional subalgebras of the algebra~$\mathfrak g$,
which is done in Section~\ref{sec:BurgersSystemLieReductionOfCodim2}.
In general, each direct Lie reduction of codimension two corresponds to several two-step reductions.
More specifically, the Lie reduction with respect to a two-dimensional algebra $\langle Q_1,Q_2\rangle$,
where $[Q_1,Q_2]=Q_2$, is equivalent to any two-step reduction from the following two families:
\begin{itemize}\itemsep=0ex
\item
the Lie reduction with respect to the subalgebra $\langle Q_2\rangle$
and the successive Lie reduction of the constructed reduced (1+1)-dimensional partial differential equation~$\mathcal R_1$
with respect to the subalgebra $\langle\tilde Q_1\rangle$,
where $\tilde Q_1$ is the Lie symmetry vector field of~$\mathcal R_1$ induced by~$Q_1$;
\item
the Lie reduction with respect to the subalgebra $\langle Q_1+cQ_2\rangle$
with an arbitrary constant~$c$
and then the nonclassical reduction of the constructed reduced (1+1)-dimensional partial differential equation~$\mathcal R_2$
with respect to the reduction operator $\tilde Q_2$ of~$\mathcal R_2$ induced by~$Q_2$.
\end{itemize}
Moreover, any two-step reduction whose first step is equivalent to one of reductions~1.8--1.11
and thus all equivalent reductions of codimension two have no sense
since reduced equations~1.8--1.11 are famous and well-studied equations,
and it suffices to use known results for them.
We can also recognize, up to simple point transformations,
some classical diffusion--convection--reaction equations
among reduced equations~1.1--1.7.

Below we discuss the reduced equations given in Table~\ref{tab:LieReductionsOfCodim1}
and present solutions of the equation~\eqref{eq:(1+2)DDegenerateBurgersEq}
which are found using the known solutions of these reduced equations.

\medskip\par\noindent
{\bf 1.7.} Reduced equation~1.7 is a member of the class of generalized Burgers equations $u_t+uu_x+f(t,x)u_{xx}=0$,
which was studied within the framework of symmetry analysis of differential equations
in \cite{poch2012a,poch2017a,poch2014a,VSL2015}; see therein for other references and a discussion of existing results.
The value $f=t^2$ corresponding to reduced equation~1.7 arises jointly with other degrees of~$t$
as a case of Lie symmetry extension within the above class.
No interesting explicit solutions are known for such values of~$f$.

\medskip\par\noindent
{\bf 1.8--1.10.}
Both reduced equations~1.8 and~1.9 coincide with the (1+1)-dimensional linear heat equation.
Reduced equation~1.10 is no other than the Burgers equation,
which is related to the heat equation via the Cole--Hopf transformation.
The corresponding invariant solutions of the equation~\eqref{eq:(1+2)DDegenerateBurgersEq}
are expressed in terms of an arbitrary solution $\theta=\theta(t,z)$
of the heat equation $\theta_t=\theta_{zz}$,
\begin{gather}\label{eq:(1+2)DDegenerateBurgersEqLieSolsInTermsOfGenSolOfHeatEq}
 1.8.\ u=\frac xt+\frac1t\theta(t,y),\quad
 1.9.\ u=\theta(t,y),\quad
1.10.\ u=-2\frac{\theta_z(t,z)}{\theta(t,z)}\mbox{ \ with \ }z=x+y.
\end{gather}
The first two solution families are generalized in Section~\ref{sec:CommonSolutionsOf(1+2)DDegenerateAndNondegenerateBurgersEqs}.
See Section~\ref{sec:ExactSolutionsOfHeatEq} for an enhanced collection of explicit solutions of the heat equation.

\medskip\par\noindent
{\bf 1.11.}
Reduced equation~1.11 is the transport equation, also called the inviscid Burgers equation.
Its general solution can be found in an implicit form,
giving a family of solutions of~\eqref{eq:(1+2)DDegenerateBurgersEq} also in an implicit form,
\[
1.11.\ F(u,x-ut)=0,
\]
where $F$ is an arbitrary nonconstant sufficiently smooth function of its arguments.

\medskip\par\noindent
{\bf 1.5.}
Reduced equation~1.5 with $\delta=0$ is the nonlinear heat equation with purely quadratic nonlinearity.
Within the symmetry analysis of differential equations,
it first appears in~\cite{doro1982a} in the course of group classification of the class of nonlinear heat equations
as a regular representative of the subclass of equations with power nonlinearities.
Exact solutions of this equation were constructed in~\cite{bara2002d} (see also \cite[Section~5.1.1.1.1]{Polyanin&Zaitsev2012}).
We use them in Section~\ref{sec:CommonSolutionsOf(1+2)DDegenerateAndNondegenerateBurgersEqs}
to construct wider families of solutions of~\eqref{eq:(1+2)DDegenerateBurgersEq} 
than the related family of $\mathfrak g_{1.5}^0$-invariant solutions.
Regular reduction operators of reduced equation~1.5 with $\delta=1$ were described in~\cite{cher2018a}.
Using Lie and nonclassical symmetries, one can construct the following inequivalent solutions for this equation:
\[
w=\frac{-2}{z_2-2z_1}, \quad
w=\frac{1\pm {\rm e}^{z_1-z_2}}{z_1}, \quad 
w=\frac1{z_1}, \quad
w=\pm {\rm e}^{z_1-z_2}.
\]
The derivation of the corresponding solutions of the (1+2)-dimensional degenerate Burgers equation~\eqref{eq:(1+2)DDegenerateBurgersEq}
\[
u=\frac{-2x}{y+\ln|x|-2t}, \quad u=\frac{x\pm {\rm e}^{t-y}}t, \quad u=\frac xt, \quad u=\pm {\rm e}^{t-y}
\]
is not significant
since a solution equivalent to the first one is obtained below in Section~\ref{sec:BurgersSystemLieReductionOfCodim2}
using a direct Lie reduction of~\eqref{eq:(1+2)DDegenerateBurgersEq} to an ordinary differential equation,
and each of the other solutions is of one of the forms~\eqref{eq:(1+2)DDegenerateBurgersEqLieSolsInTermsOfGenSolOfHeatEq}.

\medskip\par\noindent
{\bf 1.1--1.4, 1.6.}
Using the substitution $w=v^{1/2}$, we can represent reduced equations 1.1--1.4 and 1.6
in the standard form for reaction--diffusion--convection equations,
\begin{gather*}
1.1'.\  v_1=\varepsilon(v^{-1/2}v_2)_2+\tfrac12z_2v^{-1/2}v_2-2(2\kappa-1)v^{1/2}-2\kappa(\kappa-1)z_1,\\
1.2'.\  v_1=\varepsilon(v^{-1/2}v_2)_2+\tfrac12z_2v^{-1/2}v_2+2v^{1/2}+2,\\
1.3'.\  v_1=\varepsilon(v^{-1/2}v_2)_2+\tfrac12z_2v^{-1/2}v_2-2v^{1/2}-2,\\
1.4'.\  v_1=(v^{-1/2}v_2)_2+\beta v^{-1/2}v_2-2v-4\varepsilon'v^{1/2}-2,\\
1.6'.\  v_1=(v^{-1/2}v_2)_2+\delta'v^{-1/2}v_2-2\delta.
\end{gather*}
Equation~1.6$'$ with $\delta'=\delta=0$ is
the nonlinear diffusion equation with power nonlinearity of degree~$-1/2$, $v_t=(v^{-1/2}v_x)_x$.
Within the symmetry analysis of differential equations,
it first appeared in~\cite{ovsi1959a} in the course of group classification of the class of nonlinear diffusion equations
as a regular representative of the subclass of equations with power diffusivity.
It was singled out due to possessing a nonclassical generalized reduction in~\cite{amer1990,king1992a}.
See Section~\ref{sec:ExactSolutionsOfNDCEs} for known and new exact solutions of equations~1.6$'$ with $\delta=0$.
Lie symmetries, regular reduction operators and exact solutions of classes
of diffusion--convection--reaction equations including equations~1.4$'$ and~1.6$'$
were considered in~\cite{cher2018a}; see also references therein.

\noprint{
Book \cite{cher2018a}, ``Murray'' equation (2.204), Eqs. (3.77), (3.79)\\
(4.17) in the context of [19]\\ (4.116), (4.117)\\ Table 5.1 on p. 209
}

\section{Lie reductions of codimension two}\label{sec:BurgersSystemLieReductionOfCodim2}

Similarly to one-dimensional subalgebras, we construct
a complete list of two-dimensional $G$-inequivalent subalgebras of the algebra~$\mathfrak g$,
\begin{gather*}
\langle D^t,                           D^x            \rangle,\quad
\langle D^t+2\kappa D^x,               P^t            \rangle,\quad
\langle D^t+2P^x,                      P^t            \rangle,\quad
\langle D^t+4D^x,                      P^t+G^x        \rangle,\\[1ex]
\langle D^t+2\kappa D^x,               G^x            \rangle,\quad
\langle D^t+2P^x,                      G^x            \rangle,\quad
\langle D^t+3D^x,                      G^x+P^y        \rangle,\\[1ex]
\langle D^t+2\kappa D^x,               P^y            \rangle,\quad
\langle D^t+2P^x,                      P^y            \rangle,\quad
\langle D^t+2D^x+2G^x,                 P^y            \rangle,\\[1ex]
\langle D^t+2\kappa D^x,               P^x            \rangle,\quad
\langle D^t+2D^x+2G^x,                 P^x            \rangle,\\[1ex]
\langle D^x+\beta P^y,                 P^t+P^y        \rangle,\quad
\langle D^x-\delta P^y,                P^t            \rangle_{\delta\in\{0,1\}},\\[1ex]
\langle D^x-P^y,                       G^x            \rangle,\quad
\langle D^x,                           G^x            \rangle,\quad
\langle D^x-\varepsilon'P^t,           P^y            \rangle_{\varepsilon'=\pm1},\quad
\langle D^x,                           P^y            \rangle,\\[1ex]
\langle D^x-\varepsilon'P^t+\beta P^y, P^x            \rangle_{\varepsilon'=\pm1,\beta\geqslant0},\quad
\langle D^x+P^y,                       P^x            \rangle,\quad
\langle D^x,                           P^x            \rangle,\\[1ex]
\langle P^t+\delta G^x,                P^y-\delta'P^x \rangle_{\delta,\delta'\in\{0,1\}},\quad
\langle P^t+\delta G^x+\delta'P^y,     P^x            \rangle_{\delta,\delta'\in\{0,1\}},\\[1ex]
\langle G^x,                           P^y+\delta'P^x \rangle_{\delta'\in\{0,1\}},\quad
\langle G^x+P^y,                       P^x            \rangle,\quad
\langle G^x,                           P^x            \rangle,\quad
\langle P^y,                           P^x            \rangle.
\end{gather*}

The subalgebras $\langle D^x,G^x\rangle$, $\langle D^x,P^x\rangle$ and $\langle G^x,P^x\rangle$
cannot be used for constructing ansatzes for~$u$
since the matrices of the $(t,x,y)$-components of the corresponding basis elements are of rank one.

For each of the other listed subalgebras,
the presented basis $(Q_1,Q_2)$ satisfies the condition $[Q_1,Q_2]\in\langle Q_2\rangle$.
This means that the Lie reduction of the equation~\eqref{eq:(1+2)DDegenerateBurgersEq}
with respect to the subalgebra $\langle Q_1,Q_2\rangle$ is equivalent to a two-step Lie reduction.
The first step is the Lie reduction of the equation~\eqref{eq:(1+2)DDegenerateBurgersEq}
with respect to the subalgebra $\langle Q_2\rangle$ to a partial differential equation~$\mathcal E$.
The Lie-symmetry vector field~$Q_2$ of the equation~\eqref{eq:(1+2)DDegenerateBurgersEq}
induces a Lie-symmetry vector field~$\tilde Q_2$ of the equation~$\mathcal E$.
Then the second step is the Lie reduction of the equation~$\mathcal E$
with respect to the algebra $\langle\tilde Q_2\rangle$ to an ordinary differential equation.
The equation~$\mathcal E$ is one of reduced equations~1.5, 1.6$_{\delta\delta'=0}$, 1.7--1.11.
Reduced equations 1.8--1.11 are comprehensively studied within the framework of group analysis of differential equations
and wide families of their exact solutions were constructed in the literature;
see the discussion in Section~\ref{sec:BurgersSystemLiereductionsOfCodim1}.
This is why codimension-two Lie reductions of the equation~\eqref{eq:(1+2)DDegenerateBurgersEq}
with respect to the listed two-dimensional subalgebras
that contain one of the vector fields~$G^x$, $P^x$, $P^y-P^x$ and~$P^y$
are needless.

Although reduced equations~1.5--1.7 were also intensively studied,
the number of their explicit solutions presented in the literature is less than one may expect.
For a proper arrangement of solutions of the equation~\eqref{eq:(1+2)DDegenerateBurgersEq}
that are related to further Lie reductions of reduced equations~1.5--1.7,
it is worthwhile to carry out codimension-two Lie reductions of the equation~\eqref{eq:(1+2)DDegenerateBurgersEq}
with respect to the listed two-dimensional subalgebras without the above property.
It is convenient to rearrange these subalgebras into the following list:
\begin{gather*}
\mathfrak g_{2.1}           =\langle D^t,             D^x    \rangle,\quad
\mathfrak g_{2.2}^\kappa    =\langle D^x+2\kappa D^t, P^t    \rangle_{\kappa\ne0},\quad
\mathfrak g_{2.3}           =\langle D^t,             P^t    \rangle,\\[1ex]
\mathfrak g_{2.4}           =\langle D^t+2P^x,        P^t    \rangle,\quad
\mathfrak g_{2.5}           =\langle D^t+4D^x,        P^t+G^x\rangle,\quad
\mathfrak g_{2.6}           =\langle D^t+3D^x,        G^x+P^y\rangle,\\[1ex]
\smash{
\mathfrak g_{2.7}^\beta     =\langle D^x+\beta P^y,   P^t+P^y\rangle,\quad
\mathfrak g_{2.8}^{\delta''}=\langle D^x-\delta''P^y, P^t    \rangle_{\delta''\in\{0,1\}}.
}
\end{gather*}

In Table~\ref{tab:LieReductionsOfCodim2}, for each of these subalgebras,
we present an ansatz constructed for $u$, the corresponding reduced ordinary differential equation
and the optimal equivalent two-step Lie reduction,
where $1.i$ is the number of the first-step Lie reduction of codimension one
to a partial differential equation,
and $\langle\dots\rangle$ is the induced one-dimensional Lie symmetry subalgebra
of reduced equation $1.i$ to be used for the second-step Lie reduction.

\begin{table}[!ht]\footnotesize
\caption{Essential Lie reductions of codimension two}
\label{tab:LieReductionsOfCodim2}
\begin{center}
\renewcommand{\arraystretch}{1.6}
\begin{tabular}{|c|c|c|l|l|}
\hline
\!no.\!& $u$                                          & $\omega$            & \hfil Reduced equation                                                                                    & \hfil Two-step Lie reduction\\
\hline
2.1  &    $xt^{-1}\varphi$                                    & $|t|^{-1/2}y$       & $\varepsilon\varphi_{\omega\omega}-\varphi^2+\frac12\omega\varphi_\omega+\varphi=0$                       & 1.5$_{\delta=0}$, $\langle\tilde D^0\rangle$\\
2.2  &    $x|x|^{-2\kappa}\varphi$                            & $|x|^{-\kappa}y$    & $\varphi_{\omega\omega}+\kappa\omega\varphi\varphi_\omega+(2\kappa-1)\varphi^2=0$\!\!                     & 1.6$_{\delta=\delta'=0}$, $\langle\tilde D^2+2\kappa\tilde D^1\rangle$\!\\
2.3  &    $y^{-2}\varphi$                                     & $x$                 & $\varphi(\varphi_\omega-6)=0$                                                                             & 1.6$_{\delta=\delta'=0}$, $\langle\tilde D^2\rangle$\\
2.4  &    ${\rm e}^{-x}\varphi$                               & ${\rm e}^{-x/2}y$   & $\varphi_{\omega\omega}+\frac12\omega\varphi\varphi_\omega+\varphi^2=0$                                   & 1.6$_{\delta=\delta'=0}$, $\langle\tilde D^2+2\tilde P^1\rangle$\\
2.5  &    $|x-t^2/2|^{1/2}\varphi+t$                          & $|x-t^2/2|^{-1/4}y$ & $\varphi_{\omega\omega}+\frac14\varepsilon'\omega\varphi\varphi_\omega-\frac12\varepsilon'\varphi^2=1$    & 1.6$_{\delta=1,\delta'=0}$, $\langle\tilde D^2+4\tilde D^1\rangle$\!\\
2.6  &\!\!$t^{-1}|t|^{3/2}\varphi+\frac32t^{-1}x-\frac12y$\!\!& $|t|^{-3/2}(x-ty)$  & $\varepsilon\varphi_{\omega\omega}-\omega\varphi\varphi_\omega-2\varphi=\frac34\omega$                    & 1.7, $\langle\tilde D^3\rangle$\\
2.7  &    $x\varphi$                                          & $y-\beta\ln|x|-t$   & $\varphi_{\omega\omega}+\beta\varphi\varphi_\omega-\varphi^2+\varphi_\omega=0$                            & 1.6$_{\delta=0,\delta'=1}$, $\langle\tilde D^1+\beta\tilde P^2\rangle$\!\\
2.8  &    $x\varphi$                                          & $y+\delta''\ln|x|$  & $\varphi_{\omega\omega}-\delta''\varphi\varphi_\omega-\varphi^2=0$                                        & 1.6$_{\delta=\delta'=0}$, $\langle\tilde D^1-\delta''\tilde P^2\rangle$\\
\hline
\end{tabular}
\end{center}
Here $\varepsilon=\sgn t$, $\varepsilon'=\sgn(x-t^2/2)$, $\delta''\in\{0,1\}$, $\beta\in\mathbb R$,
$\varphi=\varphi(\omega)$ is the new unknown function of the invariant independent variable~$\omega$.
$\tilde D^0:=2z_1\p_{z_1}+z_2\p_{z_2}-2w\p_w$ and $\tilde D^3:=2z_1\p_{z_1}+3z_2\p_{z_2}+w\p_w$.
\end{table}

We find (nonzero) exact solutions of some of the constructed reduced equations
and present the associated exact solutions of the equation~\eqref{eq:(1+2)DDegenerateBurgersEq}.
Hereafter $C_0$, $C_1$ and $C_2$ are arbitrary constants,
and, except Case~2.8, the constant tuple $(C_1,C_2)$ is nonzero and defined up to a~nonzero multiplier whenever it appears.

\medskip\par\noindent
{\bf 2.2.}
The maximal Lie invariance algebra of reduced equation~2.2 is
$\langle\omega\p_\omega-2\varphi\p_\varphi\rangle$ since $\kappa\ne0$.
It is induced by the maximal Lie invariance algebra~$\mathfrak g$ of the initial equation~\eqref{eq:(1+2)DDegenerateBurgersEq}
and by the maximal Lie invariance algebra~$\mathfrak g$ of reduced equation 1.6$_{\delta=\delta'=0}$ as well.

Reduced equation~2.2 with $\kappa=2/3$ can be represented in the form $\varphi_{\omega\omega}+\frac13(\omega\varphi^2)_\omega=0$
and thus can be integrated once to the parameterized Riccati equation $\varphi_\omega+\frac13\omega\varphi^2=3C_0^3$,
where the integration constant is represented as $3C_0^3$ for convenience of the further integration.
At the same time, we can set $C_0\in\{0,1\}$ modulo induced scaling symmetry transformations of reduced equation~2.2.
For $C_0=0$, the above Riccati equation is directly integrated by the separation of variables to $\varphi=6/(\omega^2+C_1)$,
where $C_1\in\{-1,0,1\}$ modulo induced scalings.
For $C_0=1$, 
this equation is reduced by the substitution $\varphi=3\omega^{-1}\psi_\omega/\psi$
to the equation $\psi_{\omega\omega}-\omega^{-1}\psi_\omega=\omega\psi$
for the first derivatives of the Airy functions, and hence its general solution is
\[
\varphi=3\frac{C_1\mathop{\rm Ai}(\omega)+C_2\mathop{\rm Bi}(\omega)}{C_1\mathop{\rm Ai}'(\omega)+C_2\mathop{\rm Bi}'(\omega)},
\]
where the Airy wave functions $\mathop{\rm Ai}$ and $\mathop{\rm Bi}$ are linearly independent solutions
of the Airy equation $\tilde\psi_{\omega\omega}=\omega\tilde\psi$.
In total, this leads to the following solutions of the equation~\eqref{eq:(1+2)DDegenerateBurgersEq}:
\[
u=\frac{6x}{y^2+C_1x^{4/3}},\ C_1\in\{-1,0,1\},
\quad
u=3x^{-1/3}\frac
{C_1\mathop{\rm Ai}(x^{-2/3}y)+C_2\mathop{\rm Bi}(x^{-2/3}y)}
{C_1\mathop{\rm Ai}'(x^{-2/3}y)+C_2\mathop{\rm Bi}'(x^{-2/3}y)}.
\]

After multiplying by~$\omega$, reduced equation~2.2 with $\kappa=1$
can be represented in the form $(\omega\varphi_\omega-\varphi)_\omega+\frac12(\omega^2\varphi^2)_\omega=0$
and can hence be integrated once to the parameterized Riccati equation $\omega\varphi_\omega-\varphi+\frac12\omega^2\varphi^2=2C_0$.
We can set $C_0\in\{-1,0,1\}$ modulo induced scaling symmetry transformations of reduced equation~2.2.
The differential substitution $\varphi=\omega^{-1}\psi_\omega/\psi$ reduces the latter equation
to the second-order linear ordinary differential equation
$\omega\psi_{\omega\omega}-2\psi_\omega-C_0\omega\psi=0$.
Integrating the derived equation depending on values of~$C_0$, $0$, $1$ and $-1$, respectively,
and substituting the results of the integration into the above differential substitution,
we construct all inequivalent solutions of reduced equation~2.2 with $\kappa=1$,
\begin{gather*}
\varphi=\frac{6\omega}{\omega^3+C_1},\ C_1\in\{0,1\},\quad
\varphi=2\frac{C_1{\rm e}^\omega-C_2{\rm e}^{-\omega}}{C_1{\rm e}^\omega(\omega-1)+C_2{\rm e}^{-\omega}(\omega+1)},\\[1ex]
\varphi=2\frac{-C_1\sin\omega+C_2\cos\omega}{C_1(\omega\cos\omega-\sin\omega)+C_2(\omega\sin\omega+\cos\omega)}.
\end{gather*}
The corresponding solutions of the equation~\eqref{eq:(1+2)DDegenerateBurgersEq} are
\begin{gather*}
u=\frac{6xy}{C_1x^3+y^3},\ C_1\in\{0,1\},\quad
u=2\frac{C_1{\rm e}^{y/x}-C_2{\rm e}^{-y/x}}{C_1{\rm e}^{y/x}(y-x)-C_2{\rm e}^{-y/x}(y+x)},\\[1ex]
u=2\frac{-C_1\sin(y/x)+C_2\cos(y/x)}{C_1(y\cos(y/x)-x\sin(y/x))+C_2(y\sin(y/x)+x\cos(y/x))}.
\end{gather*}

\medskip\par\noindent
\noindent
{\bf 2.3.}
Reduced equation~2.3 trivially integrates to $\varphi=6(\omega+C_0)$.
The associated solution $u=6(x+C_0)y^{-2}$ of the equation~\eqref{eq:(1+2)DDegenerateBurgersEq},
where $C_0=0\bmod G$, is significantly generalized
in Section~\ref{sec:CommonSolutionsOf(1+2)DDegenerateAndNondegenerateBurgersEqs}
to the solution~\eqref{eq:(1+2)DDegenerateBurgersEqSolutionSetIn1SolutionOfLHEq}.

\medskip\par\noindent
{\bf 2.7.}
For any~$\beta$,
the maximal Lie invariance algebra of reduced equation~2.7 is $\langle\p_\omega\rangle$,
and it is induced by the maximal Lie invariance algebra~$\mathfrak g$ of the equation~\eqref{eq:(1+2)DDegenerateBurgersEq}
and by the maximal Lie invariance algebra of reduced equation 1.6$_{\delta=0,\delta'=1}$ as well.

We multiply reduced equation~2.7 with $\beta=-2$ by ${\rm e}^\omega$ and then integrate once,
successively deriving $({\rm e}^\omega\varphi_\omega)_\omega=({\rm e}^\omega\varphi^2)_\omega$ and
$\varphi_\omega=\varphi^2+C_0{\rm e}^{-\omega}$.
We can set $C_0\in\big\{-\frac14,0,\frac14\big\}$ modulo the induced shift symmetry transformations of reduced equation~2.7.
Then we use the standard differential substitution $\varphi=-\psi_\omega/\psi$
to reduce the later, Riccati, equation to the second-order linear ordinary differential equation
$\psi_{\omega\omega}+C_0{\rm e}^{-\omega}\psi=0$, which is solved in terms of
elementary functions, Bessel functions~$J$ and~$Y$ or modified Bessel functions~$I$ and~$K$ if $C_0=0,-\frac14,\frac14$, respectively.
The computation results in the complete set of inequivalent solutions of reduced equation~2.7 with $\beta=-2$,
\begin{gather*}
\varphi=-\frac1\omega,\quad
\varphi=-\frac12{\rm e}^{-\omega/2}\frac{C_1J_1({\rm e}^{-\omega/2})+C_2Y_1({\rm e}^{-\omega/2})}{C_1J_0({\rm e}^{-\omega/2})+C_2Y_0({\rm e}^{-\omega/2})},\\[1ex]
\varphi= \frac12{\rm e}^{-\omega/2}\frac{C_1I_1({\rm e}^{-\omega/2})-C_2K_1({\rm e}^{-\omega/2})}{C_1I_0({\rm e}^{-\omega/2})+C_2K_0({\rm e}^{-\omega/2})}.
\end{gather*}
Using these solutions and ansatz~2.7, we construct explicit solutions of the equation~\eqref{eq:(1+2)DDegenerateBurgersEq},
\begin{gather*}
u=-\frac x{y+2\ln|x|-t},\quad
u=-\frac12{\rm e}^{(t-y)/2}\frac{C_1J_1(x^{-1}{\rm e}^{(t-y)/2})+C_2Y_1(x^{-1}{\rm e}^{(t-y)/2})}{C_1J_0(x^{-1}{\rm e}^{(t-y)/2})+C_2Y_0(x^{-1}{\rm e}^{(t-y)/2})},\\[1ex]
u= \frac12{\rm e}^{(t-y)/2}\frac{C_1I_1(x^{-1}{\rm e}^{(t-y)/2})-C_2K_1(x^{-1}{\rm e}^{(t-y)/2})}{C_1I_0(x^{-1}{\rm e}^{(t-y)/2})+C_2K_0(x^{-1}{\rm e}^{(t-y)/2})}.
\end{gather*}

\medskip\par\noindent
{\bf 2.8.}
Reduced equation~2.8 with $\delta''=0$ is $\varphi_{\omega\omega}=\varphi^2$.
Its maximal Lie invariance algebra $\langle\omega\p_\omega-2\varphi\p_\varphi,\p_\omega\rangle$ is induced
by the maximal Lie invariance algebra~$\mathfrak g$ of the initial equation~\eqref{eq:(1+2)DDegenerateBurgersEq}
and by the maximal Lie invariance algebra of reduced equation 1.6$_{\delta=\delta'=0}$ as well.
The general solution of reduced equation~2.8 with $\delta''=0$ is $\varphi=\wp\left(\omega/\sqrt6+C_2;0,C_1\right)$.
Here $\wp=\wp(z;\textsl{g}_2,\textsl{g}_3)$ is the Weierstrass elliptic function
with the invariants $\textsl{g}_2$ and $\textsl{g}_3$, which are respectively equal to $0$ and~$C_1$ for~$\varphi$. 
This function satisfies the differential equation $(\wp_z)^2 = 4\wp^3-\textsl{g}_2\wp-\textsl{g}_3$.
The corresponding solution \mbox{$u=\wp\big(y/\sqrt6+C_2;0,C_1\big)x$} of the equation~\eqref{eq:(1+2)DDegenerateBurgersEq}
is a particular case of the more general solution~\eqref{eq:(1+2)DDegenerateBurgersEqSolutionsWithLameFunctions} with $\varphi=0$.

\begin{remark}
The order of reduced equations 2.2, 2.4, 2.7 and~2.8 can be lowered using their Lie symmetries,
which leads to Abel equations of the second kind.
(The maximal Lie invariance algebras of the other second-order reduced equations from Table~\ref{tab:LieReductionsOfCodim2}
are zero.)
Thus, reduced equations 2.2 and~2.4 are of the same general form
$\varphi_{\omega\omega}+\alpha\omega\varphi\varphi_\omega+\beta\varphi^2=0$,
where $\alpha$ and~$\beta$ are constants.
The point transformation $\tilde\omega=\ln|\omega|$, $\tilde\varphi=\omega^2\varphi$
hinted by the Lie symmetry vector field $\omega\p_\omega-2\varphi\p_\varphi$ reduces
the last equation to the autonomous equation
\[\tilde\varphi_{\tilde\omega\tilde\omega}-3\tilde\varphi_{\tilde\omega}+6\tilde\varphi
+\alpha\tilde\varphi\tilde\varphi_{\tilde\omega}+(\beta-2\alpha)\tilde\varphi^2=0.\]
The further standard substitution $p(z)=\tilde\varphi_{\tilde\omega}$, $z=\tilde\varphi$
leads to the equation $pp_z-3p+6z+\alpha zp+(\beta-2\alpha)z^2=0$.
Reduced equations~2.7 and~2.8 themselves are autonomous, and the above substitution reduces them to
the equations $pp_z+\beta zp-z^2+p=0$ and $pp_z-\delta''zp-z^2=0$, respectively.
\end{remark}

\section{Generalized symmetries}\label{sec:(1+2)DDegenerateBurgersEqGenSyms}

Generalized symmetries of the (1+2)-dimensional degenerate Burgers equation~\eqref{eq:(1+2)DDegenerateBurgersEq}
are much poorer than those of the (1+1)-dimensional classical or inviscid Burgers equations.

\begin{theorem}\label{thm:(1+2)DDegenerateBurgersEqGenSyms}
The quotient algebra of generalized symmetries of the (1+2)-dimensional degenerate Burgers equation~\eqref{eq:(1+2)DDegenerateBurgersEq}
by the subalgebra of trivial generalized symmetries%
\footnote{\label{fnt:DefOfTrivialGenSyms}%
A generalized symmetry of a system of differential equations~$\mathcal L$ is called \emph{trivial} 
if its characteristic vanishes on the solution set of~$\mathcal L$; see, e.g., \cite[Section~5.1]{olve1993b}.
}
of this equation 
is naturally isomorphic to its maximal Lie invariance algebra~$\mathfrak g$.
\end{theorem}

The rest of this section deals with the proof of Theorem~\ref{thm:(1+2)DDegenerateBurgersEqGenSyms}.

In addition to the notation $(t,x,y)$,
it is convenient to simultaneously use another notation for the independent variables, $z_0=t$, $z_1=x$, $z_2=y$,
and thus $z=(z_0,z_1,z_2)=(t,x,y)$ is the tuple of independent variables.
A~differential function $F=F[u]$ of~$u$ is, roughly speaking, a~smooth function of~$z$, $u$
and a~finite number of derivatives of~$u$ with respect to~$z$.
Since the equation~\eqref{eq:(1+2)DDegenerateBurgersEq} is of evolution type,
each of its generalized symmetries is equivalent to a generalized symmetry
in the evolution form $\eta[u]\p_u$ \cite[Section~5.1]{olve1993b},
where the characteristic $\eta=\eta[u]$ of this symmetry
does not depend on derivatives of~$u$ involving differentiation with respect to~$z_0$.
Therefore, it suffices to consider only such generalized symmetries,
which are called reduced generalized symmetries.
The characteristic $\eta=\eta[u]$ of any of them satisfies the equation
\begin{gather}\label{eq:(1+2)DDegenerateBurgersEqConditionForGenSyms}
\mathrm D_0\eta+u\mathrm D_1\eta+u_x\eta-\mathrm D_2^2\eta=0
\end{gather}
on the manifold defined by the equation~\eqref{eq:(1+2)DDegenerateBurgersEq} and its differential consequences
in the corresponding jet space $\mathrm J^{\infty}(\mathbb R^3_z\times\mathbb R_u)$, 
and any solution of~\eqref{eq:(1+2)DDegenerateBurgersEqConditionForGenSyms} 
leads to a generalized symmetry of~\eqref{eq:(1+2)DDegenerateBurgersEq}.

Here and in what follows $\mathrm D_\mu$ denotes the operator of total derivative with respect to~$z_\mu$,
$\mathrm D_\mu=\p_\mu+u_{\alpha+\epsilon_\mu}\p_{u_\alpha}$.
The index~$\mu$ runs from~0 to~2, 
the multi-index $\alpha=(\alpha_0,\alpha_1,\alpha_2)$ runs through $\mathbb N_0{}^3$, $\mathbb N_0:=\mathbb N\cup\{0\}$, 
the conventional summation over repeated indices is used, 
and $\epsilon_\mu$ is the index triple with zeros everywhere except the $\mu$th entry that equals $1$.
The coordinate~$u_\alpha$ of the jet space $\mathrm J^{\infty}(\mathbb R^3_z\times\mathbb R_u)$ 
is identified with the derivative of~$u$ of order~$\alpha$,
$u_\alpha=\p^{|\alpha|}u/\p z_0^{\alpha_0}\p z_1^{\alpha_1}\p z_2^{\alpha_2}$.
For a $p$-tuple $\beta=(\beta^1,\dots,\beta^p)\in\mathbb N_0{}^p$, $p\in\mathbb N$, 
we denote $\#\beta:=p$ and $|\beta|:=\beta^1+\dots+\beta^p$, 
and $\#\beta:=0$ if the tuple~$\beta$ is empty.
The order of a differential function~$F$ is equal to the highest order of jet variables involved in~$F$,
where $\ord z_\mu=-\infty$ and $\ord u_\alpha=|\alpha|$.
We also use a shorter notation for coordinates of the jet space $\mathrm J^{\infty}(\mathbb R^3_z\times\mathbb R_u)$
associated with derivatives of~$u$ that do not involve differentiation with respect to~$z_0$,
$u_{kl}:=\p^{k+l}u/\p x^k\p y^l$, $k,l\in\mathbb N_0$,
and $\eta^{kl}:=\eta_{u_{kl}}$.
By $m$ and $n$ we denote the maximal orders of arguments of~$\eta$ with respect to~$x$ and~$y$, respectively,
i.e.,
$\eta=\eta(t,x,y,u_{kl},0\leqslant k\leqslant m,0\leqslant l\leqslant n)$, 
where $\eta_{u_{ml}}\ne0$ and $\eta_{u_{kn}}\ne0$ for some $k\in\{0,\dots,m\}$ and $l\in\{0,\dots,n\}$.
By default,
the indices~$k$, $k'$ and~$\kappa$ run from~0 to~$m$, and 
the indices~$l$, $l'$ and~$\lambda$ run from~0 to~$n$
whenever different ranges or additional constraints are not indicated explicitly.

After expanding the equation~\eqref{eq:(1+2)DDegenerateBurgersEqConditionForGenSyms} and
substituting the expressions implied by the equation~\eqref{eq:(1+2)DDegenerateBurgersEq}
and its differential consequences for derivatives $u_\alpha$ with $\alpha_0=1$,
we obtain the equation
\begin{gather}\label{eq:(1+2)DDegenerateBurgersEqExpandedConditionForGenSyms2}
\begin{split}
R:={}&\eta_t+\eta u_{10}+\eta_xu_{00}-\eta_{yy}-2\eta_{yu_{kl}}u_{k,l+1}-\eta_{u_{kl}u_{k'l'}}u_{k,l+1}u_{k',l'+1}
\\
&-\mathop{\mbox{$\displaystyle\sum_{k'\leqslant k,\,l'\leqslant l}$}}\limits_{(k',l')\ne(0,0)}
\binom k{k'}\binom l{l'}\eta_{u_{kl}}u_{k'l'}u_{k-k'+1,l-l'}=0.
\end{split}
\end{gather}

We will show that the characteristic~$\eta$ is a polynomial in the variables~$u_{kl}$
with coefficients depending on $(t,x,y)$.
For this purpose, we differentiate the equation~\eqref{eq:(1+2)DDegenerateBurgersEqExpandedConditionForGenSyms2}
with respect to~$u_{\kappa,\lambda+1}$ with an arbitrary fixed $(\kappa,\lambda)$, obtaining
\begin{gather}\label{eq:(1+2)DDegenerateBurgersEqExpandedConditionForGenSyms}
\begin{split}
2\mathrm D_2\eta^{\kappa\lambda}={}&\eta^{\kappa,\lambda+1}_t+\eta^{\kappa,\lambda+1}u_{10}+\eta^{\kappa,\lambda+1}_xu_{00}-\eta^{\kappa,\lambda+1}_{yy}
-2\eta^{\kappa,\lambda+1}_{yu_{kl}}u_{k,l+1}
\\&
-\eta^{\kappa,\lambda+1}_{u_{kl}u_{k'l'}}u_{k,l+1}u_{k',l'+1}
-\mathop{\mbox{$\displaystyle\sum_{k'\leqslant k,\,l'\leqslant l}$}}\limits_{(k',l')\ne(0,0)}
\binom k{k'}\binom l{l'}\eta^{\kappa,\lambda+1}_{u_{kl}}u_{k'l'}u_{k-k'+1,l-l'}
\\
&-\sum_{k\geqslant\kappa,\,l\geqslant\lambda+1}
\frac{k+1}{k-\kappa+1}\binom k{\kappa}\binom l{\lambda+1}
\eta^{kl}u_{k-\kappa+1,l-\lambda-1}
\\
&-(1-\delta_{\kappa0})\sum_{l>\lambda+1}  \binom l{\lambda+1}  
\eta^{\kappa-1,l}u_{0,l-\lambda-1},
\end{split}
\end{gather}
where $\delta$ denotes the Kronecker delta.
We consider the collection of equations~\eqref{eq:(1+2)DDegenerateBurgersEqExpandedConditionForGenSyms}
for all the admitted values of~$(\kappa,\lambda)$ as a system of differential equations with respect to $\eta^{kl}$
on the jet space with the independent variable~$y$ and dependent variables $u_{k0}$, $k\in\mathbb N_0$,
where $t$ and~$x$ rather play the role of parameters.
We successively integrate this system, starting from $\lambda=n$ and stepping down the value of~$\lambda$.
The equations~\eqref{eq:(1+2)DDegenerateBurgersEqExpandedConditionForGenSyms} with $\lambda=n$
are $2\mathrm D_2\eta^{\kappa n}=0$, which means that $\eta^{\kappa n}$ depends only on~$(t,x)$.
For each subsequent value of~$\lambda$, the equations~\eqref{eq:(1+2)DDegenerateBurgersEqExpandedConditionForGenSyms}
take the forms $2\mathrm D_2\eta^{\kappa\lambda}=f^{\kappa\lambda}$,
where $f^{\kappa\lambda}$ is a differential function in the above jet space,
which is expressed via~$\eta^{kl}$ with $l>\lambda$.
The formula (5.151) and Theorem~5.104 of~\cite{olve1993b} imply
that $\eta^{kl}$ is a polynomial in the tuple $(u_{kl})$ with coefficients depending on $(t,x,y)$
if $f^{\kappa\lambda}$ is such a polynomial.
Therefore, it is easy to prove by induction with respect to stepped down~$\lambda$
that $\eta^{kl}$ is such a polynomial.
Since $\eta^{kl}:=\eta_{u_{kl}}$, the characteristic~$\eta$ is also such a polynomial,
and thus $R$ is of the same kind.

It is convenient to introduce the set $\mathcal I$ of \emph{unordered} tuples of pairs of nonnegative integers, i.e.,
the set $\bigcup_{p=0}^\infty(\mathbb N_0\times\mathbb N_0)^p$ modulo
the equivalence of tuples of the same length~$p$ with respect to permutation of their elements.
The length of an element~$\bar\iota$ of~$\mathcal I$ (and, similarly, of other tuples),
the ``projections'' of~$\bar\iota$ to the $x$- and $y$-directions
and the corresponding monomial of jet variables $u_{kl}$
are denoted by~$\#\bar\iota$, $\pi_x\bar\iota$, $\pi_y\bar\iota$ and $\bar u_{\bar\iota}$, respectively,
$\#\bar\iota:=p$, $\pi_x\bar\iota:=\bar k=(k_1,\dots,k_p)$, $\pi_y\bar\iota:=\bar l=(l_1,\dots,l_p)$,
$\bar u_{\bar\iota}:=\prod_{s=1}^pu_{k_sl_s}$ if $\bar\iota=\big((k_1,l_1),\dots,(k_p,l_p)\big)$ with $p>0$
and $\#\bar\iota:=0$, $\pi_x\bar\iota=\pi_y\bar\iota=()$, $\bar u_{\bar\iota}:=1$ if the tuple~$\bar\iota$ is empty.
We write $\bar\iota=\bar k\otimes\bar l$ if $\pi_x\bar\iota=\bar k$ and $\pi_y\bar\iota=\bar l$.

In the above notation, the characteristic~$\eta$ can be written down as
\[
\eta=\sum_{\bar\iota\in\mathcal I}\eta^{\bar\iota}(t,x,y)\bar u_{\bar\iota}=\sum_{\bar\iota\in I}\eta^{\bar\iota}(t,x,y)\bar u_{\bar\iota}, \quad
\]
where
the coefficients $\eta^{\bar\iota}=\eta^{\bar\iota}(t,x,y)$ are smooth functions of~$(t,x,y)$, 
which do not vanish if and only if $\bar\iota\in I$,
and~$I$ is a finite subset of~$\mathcal I$.
Note that the degree $\deg\bar u_{\bar\iota}$ of a monomial~$\bar u_{\bar\iota}$ in totality of the jet variables $u_{kl}$
coincides with~$\#\bar\iota$.
Hence the similar degree $\deg\eta$ of~$\eta$ is naturally equal to $r:=\max\{\#\bar\iota\mid\bar\iota\in I\}$.

We define an appropriate ranking of the monomials of jet variables $u_{kl}$ based on the degree of monomials.
Ranking the monomials of jet variables $u_{kl}$ is equivalent to ranking the elements of~$\mathcal I$.
Since monomials are defined up to permutations of multipliers within them as well as the corresponding index tuples are considered to be unordered,
we preliminarily represent each $\bar\iota\in \mathcal I$ with $p:=\#\bar\iota>0$ in the form $\bar\iota=\big((k_1,l_1),\dots,(k_p,l_p)\big)$,
where $k_1\geqslant\dots\geqslant k_p$ and $l_s\geqslant l_{s+1}$ for all $s\in\{1,\dots,p-1\}$ with $k_s=k_{s+1}$.
We set that $\bar u_{\bar\iota}$ is of less ranking than $\bar u_{\bar\iota'}$ if
\[
\big(\#\bar\iota,\pi_x\bar\iota,\pi_y\bar\iota\big)\prec  
\big(\#\bar\iota',\pi_x\bar\iota',\pi_y\bar\iota'\big),   
\]
where the symbol~$\prec$ denotes the lexicographic order induced by the natural order $<$ of the integers,
$\pi_y\bar\iota$ is compared with $\pi_y\bar\iota'$ with respect to the lexicographic order,
and then
$\pi_x\bar\iota$ is compared with~$\pi_x\bar\iota'$ also with respect to the lexicographic order.

Hereafter, for an appropriate~$p\in\mathbb N$,
$\bar e_\sigma$ denotes the $p$-tuple whose components are zeros except $\sigma$th component equal to~1,
and $\bar0$ is the $p$-tuple of zeros.

\begin{lemma}\label{lem:(1+2)DDegenerateBurgersEGenSyms}
Suppose that for a finite $I'\subset \mathcal I$,
a polynomial $\theta=\sum_{\bar\iota\in I'}\theta^{\bar\iota}(t,x,y)\bar u_{\bar\iota}$ in $(u_{kl})$
satisfies the equation
\begin{gather}\label{eq:(1+2)DDegenerateBurgersEqExpandedConditionForGenSyms3}
V:=\theta u_{10}
-\mathop{\mbox{$\displaystyle\sum_{k'\leqslant k,\,l'\leqslant l}$}}\limits_{(k',l')\ne(0,0)}
\binom k{k'}\binom l{l'}\theta_{u_{kl}}u_{k'l'}u_{k-k'+1,l-l'}=0.
\end{gather}
Then for each fixed $p\in\mathbb N$ we have
$\{\bar\iota\in I'\mid\#\bar\iota=p,\,\theta^{\bar\iota}\ne0\}\subseteq\{\bar e_1\otimes\bar0,\,\bar0\otimes\bar e_1\}$
with $\#\bar e_1=\#\bar0=p$.
More specifically, the monomials of degree~$p$ in~$\theta$ are exhausted by
$\theta^{\bar0\otimes\bar e_1}u_{01}u_{00}^{p-1}+\theta^{\bar e_1\otimes\bar0}u_{10}u_{00}^{p-1}$.
\end{lemma}

\begin{proof}
The left-hand side~$V$ of the equation~\eqref{eq:(1+2)DDegenerateBurgersEqExpandedConditionForGenSyms3}
is a polynomial in~$(u_{kl})$.
Collecting coefficients of various monomials of~$(u_{kl})$ in~\eqref{eq:(1+2)DDegenerateBurgersEqExpandedConditionForGenSyms3},
we obtain a system for the coefficients~$\theta^{\bar\iota}$, $\bar\iota\in I'$.
This system splits into uncoupled subsystems associated with fixed values of $\#\bar\iota$ and $|\pi_y\bar\iota|$.
This is why we can partition the monomials of~$\theta$ into subsets according to these values and assume without loss of generality
that the set~$I'$ consists of tuples $\bar\iota$ with the same value $(p,\nu)$ of $(\#\bar\iota,|\pi_y\bar\iota|)$.
Then all monomials in~$V$ are of degree~$p+1$, i.e.,
$\#\bar\iota=p+1$ and $|\pi_y\bar\iota|=\nu$ for the values of~$\bar\iota$ associated with monomials of~$V$.

Let $\theta^{\bar\iota_*}\bar u_{\bar\iota_*}$ be the leading monomial of~$\theta$ with respect to the above ranking,
and thus $\theta^{\bar\iota_*}\ne0$.
Denote $\bar\kappa:=\pi_x\bar\iota_*$ and $\bar\lambda:=\pi_y\bar\iota_*$, so $|\bar\lambda|=\nu$,
and assume that the index~$s$ runs from~1 to~$p$.
Then the leading monomial in~$V$ is $u_{\bar\iota}=u_{\bar\iota_*}u_{10}$, i.e., $\bar\iota=(\bar\kappa,1)\otimes(\bar\lambda,0)$.
It arises in the first summand in~$V$ and the summands associated with the values
$(k',l')=(\kappa_s,\lambda_s)$ if $(\kappa_s,\lambda_s)\ne(0,0)$ and
$(k',l')=(1,0)$ if $\kappa_s>1$ or $\lambda_s\ne0$.
(The last two inequalities are required in order to avoid doubly counting the monomials of~$V$, 
which could happen for $(\kappa_s,\lambda_s)=(1,0)$.)
Collecting coefficients of this monomial in~$V$ leads to the equation
\[
\theta^{\bar\iota_*}\Bigg(1-
\sum_{s\colon (\kappa_s,\lambda_s)\ne(0,0)}1-\sum_{s\colon \kappa_s>1}\kappa_s-\sum_{s\colon\kappa_s=1,\,\lambda_s\ne0}1
\Bigg)=0.
\]
Since $\theta^{\bar\iota_*}\ne0$, this equation can be satisfied only if
either $\bar\iota_*=\bar e_1\otimes\bar0$ or $\bar\iota_*=\bar0\otimes\nu\bar e_1$ with $\nu>0$.

If $\bar\iota_*=\bar e_1\otimes\bar0$, then $\nu=0$ and thus
$\theta=\theta^{\bar e_1\otimes\bar0}u_{10}u_{00}^{p-1}+\theta^{\bar0\otimes\bar0}u_{00}^p$.
The polynomial~$V$ reduces to $V=\theta^{\bar0\otimes\bar0}u_{10}u_{00}^p$, i.e., $\theta^{\bar0\otimes\bar0}=0$.

Suppose that $\bar\iota_*=\bar0\otimes\nu\bar e_1$ with $\nu>0$.
The maximality of $\bar\iota_*$ in $I'$ implies that $|\pi_x\bar\iota|=0$ for all $\bar\iota\in I'$.
\noprint{
If $\mu:=\max\{|\pi_x\bar\iota|\mid\bar\iota\in I',\,\theta^{\bar\iota}\ne0\}$,
then we consider $\bar\iota_0=\max\{\bar\iota\in I'\mid |\pi_x\bar\iota|=\mu\}$
and re-denote $\bar\kappa:=\pi_x\bar\iota_0$ and $\bar\lambda:=\pi_y\bar\iota_0$.
Following the above arguments, we derive the equation~\eqref{eq:(1+2)DDegenerateBurgersEqExpandedConditionForGenSyms3},
where  $\theta^{\bar\iota_*}$ is replaced by~$\theta^{\bar\iota_0}$.
Since $|\pi_x\bar\iota|=\nu>0$, this equation implies the condition $\mu=0$.
}
In view of the assumption that $|\pi_y\bar\iota|:=\nu$ for all $\bar\iota\in I'$,
the polynomial~$\theta$ takes the form $\theta=\sum_{\bar l\colon |\bar l|=\nu}\theta^{\bar0\otimes\bar l} \bar u_{\bar0\otimes\bar l}$
and thus
\[
V=\sum_{\bar l\colon |\bar l|=\nu}\Bigg(
 \theta^{\bar0\otimes\bar l}  \bar u_{\bar0\otimes\bar l}u_{10}
-\theta^{\bar0\otimes\bar l}\sum_{s\colon l_s\ne0}\sum_{l'=1}^{l_s}\binom{l_s}{l'}\bar u_{\bar0\otimes(\bar l-(l_s-l')\bar e_s)}u_{1,l_s-l'}
\Bigg).
\]
If we suppose that $\nu>1$, then there is only one monomial in~$V$ containing~$u_{1,\nu-1}$,
which corresponds to $\bar l=\nu\bar e_1$, $s=1$ and $l'=1$,
and hence $\theta^{\bar0\otimes\nu\bar e_1}=0$ although it should be nonzero as the leading coefficient.
This contradiction implies that $\nu\leqslant1$.
If $\nu=1$, then $\theta=\theta^{\bar0\otimes\bar e_1} \bar u_{\bar0\otimes\bar e_1}$ and $V\equiv0$.
For $\nu=0$, we derive from the equation~\eqref{eq:(1+2)DDegenerateBurgersEqExpandedConditionForGenSyms3} 
that $\theta=0$. 
\end{proof}

Now we start splitting the equation~\eqref{eq:(1+2)DDegenerateBurgersEqExpandedConditionForGenSyms2}
via collecting coefficients of various monomials in its left-hand side~$R$.

Let $\mu:=\max\{|\pi_x\bar\iota|\mid\bar\iota\in I\}$.
We separate the monomials in~$R$ containing~$\bar u_{\bar\iota}$,
where the value of $|\pi_x\bar\iota|$ is maximum among the monomials of~$R$, which is equal to $\mu+1$.
If $\mu\geqslant1$, then this leads to the equation~\eqref{eq:(1+2)DDegenerateBurgersEqExpandedConditionForGenSyms3},
where $\theta$ is the sum of the monomials $\eta^{\bar\iota}\check u_{\bar\iota}$ for $\bar\iota\in I$ with $|\pi_x\bar\iota|=\mu$.
Then Lemma~\ref{lem:(1+2)DDegenerateBurgersEGenSyms} implies 
that $\mu=1$ and $|\pi_y\bar\iota|=0$ for any of these~$\bar\iota$. 

Collecting the monomials of~$R$ with $|\pi_x\bar\iota|=1$ and $|\pi_y\check\iota|\geqslant3$
gives the equation~\eqref{eq:(1+2)DDegenerateBurgersEqExpandedConditionForGenSyms3},    
where $\theta$ is the sum of the monomials $\eta^{\bar\iota}\check u_{\bar\iota}$ 
with $|\pi_x\bar\iota|=0$ and $|\pi_y\check\iota|\geqslant3$.
In view of Lemma~\ref{lem:(1+2)DDegenerateBurgersEGenSyms}, $\eta^{\bar\iota}=0$ for any of these~$\bar\iota$.

Therefore, $I=
\{\bar\iota\in\mathcal I\mid |\pi_x\bar\iota|=1,\,|\pi_y\bar\iota|=0\}\cup
\{\bar\iota\in\mathcal I\mid |\pi_x\bar\iota|=0,\,|\pi_y\bar\iota|\leqslant2\}$.
After changing the notation of the coefficients and jet variables, we can represent~$\eta$ in the form 
\[
\eta=\sum_{p=0}^\infty(\zeta^pu_x+\chi^{2p}u_{yy}+\chi^{1p}u_y+\chi^{0p}+\varphi^pu_y{}^2)u^p, \quad
\]
where the coefficients $\zeta^p$, $\chi^{2p}$, $\chi^{1p}$, $\chi^{0p}$ and~$\varphi^p$, $p\in\mathbb N_0$,
are smooth functions of~$(t,x,y)$, only a finite number of which do not vanish.
We substitute this representation for~$\eta$ 
into the equation~\eqref{eq:(1+2)DDegenerateBurgersEqExpandedConditionForGenSyms2}.
Successively collecting coefficients of the monomials 
$u_{yyy}u_yu^{p-1}$ with $p\geqslant1$, 
$u_{yy}{}^2u^p$ with $p\geqslant0$, 
$u_xu_y{}^2u^{p-2}$ with $p\geqslant2$, 
$u_{yy}u_yu^{p-1}$ with $p\geqslant1$ and
$u_y{}^2u^{p-2}$ with $p\geqslant2$ 
in the resulting equation, 
we respectively derive the equations 
\[
\chi^{2p}=0,\ p\geqslant1,\quad
\varphi^p=0,\ p\geqslant0,\quad
\zeta^p=0,\ p\geqslant2,\quad
\chi^{1p}=0,\ p\geqslant1,\quad
\chi^{0p}=0,\ p\geqslant2.
\]
They substantially restrict the form of~$\eta$ to 
$\eta=\zeta^1uu_x+\zeta^0u_x+\chi^{20}u_{yy}+\chi^{10}u_y+\chi^{01}u+\chi^{00}$.
The system of determining equations for the remaining coefficients of~$\eta$, 
which follows from the equation~\eqref{eq:(1+2)DDegenerateBurgersEqExpandedConditionForGenSyms2}, 
can be simplified to
\begin{gather*}
\zeta^1_x=\zeta^1_y=0,\quad \zeta^0_y=0,\quad 
\chi^{20}=-\zeta^1,\quad 
\chi^{01}=-\zeta^1_t-\zeta^0_x,\quad
\chi^{00}=-\zeta^0_t,\\
\chi^{10}_x=0,\quad
\chi^{01}_x=0,\quad
2\chi^{10}_y=-\zeta^1_t,\quad 
\chi^{10}_t=2\chi^{01}_y,\quad  
\chi^{01}_t=0,\quad  
\chi^{00}_t=0.  
\end{gather*}
Any solution of this system corresponds to $\eta$ that is equivalent to the characteristic 
of a Lie symmetry of the equation~\eqref{eq:(1+2)DDegenerateBurgersEq}, 
which completes the proof of Theorem~\ref{thm:(1+2)DDegenerateBurgersEqGenSyms}.

The computation of the last step can be checked via posing the upper bound two on the order of~$\eta$
as a differential function
and using, e.g., the excellent package {\sf Jets} by H.~Baran and M.~Marvan \cite{BaranMarvan,Marvan2009}
for {\sf Maple}.

\begin{remark}
Although the equation~\eqref{eq:(1+2)DDegenerateBurgersEq} admits no genuine generalized symmetries,
this is not the case at least for reduced equations~1.8--1.11.
In other words, the equation~\eqref{eq:(1+2)DDegenerateBurgersEq} admits
wide families of hidden generalized symmetries associated with reduced equations~1.8--1.11.
Recall that reduced equations~1.8 and~1.9 coincide with the (1+1)-dimensional linear heat equation,
and reduced equations~1.10 and~1.11 are the Burgers equation and the transport equation, respectively.
Generalized symmetries of the heat equation and the Burgers equation are well known,
see, e.g., \cite[Example~5.21]{olve1993b} and \cite[Chapter 4, Section 4.2]{boch1999a}.
In particular, the quotient algebra of generalized symmetries of the heat equation $w_1=w_{22}$
by the subalgebra of trivial generalized symmetries of this equation
is naturally isomorphic to the algebra
\[
\big\langle \big(t\bar{\mathrm D}_{z_2}+\tfrac12x\big)^nw^{(m)}\p_w,\,n,m\in\mathbb N_0, \ f(z_1,z_2)\p_w\big\rangle,
\]
where $\omega^{(m)}:=\p_2^{\;m}w$,
$\bar{\mathrm D}_{z_2}:=\p_2+\sum_{k=0}^\infty\omega^{(k+1)}\p_{\omega^{(k)}}$
is the reduced operator of total derivative with respect to~$z_2$,
and the function $f=f(z_1,z_2)$ runs through the solution set of the heat equation $f_{z_1}=f_{z_2z_2}$.
The algebra of generalized symmetries of the transport equation $w_1+ww_2=0$
as well as the spaces of cosymmetries and of conservation laws of this equation
are exhaustively described in the end of the next Section~\ref{sec:(1+2)DDegenerateBurgersEqCLs}.
\end{remark}

\section{Cosymmetries and conservation laws}\label{sec:(1+2)DDegenerateBurgersEqCLs}

In contrast to the classical Burgers equation whose space of local conservation laws is one-dimensional,
the (1+2)-dimensional degenerate Burgers equation~\eqref{eq:(1+2)DDegenerateBurgersEq}
admits an infinite-dimensional space of local conservation laws.
(See \cite{boch1999a,olve1993b,popo2008a,vino1984a} for definitions of related notions and necessary theoretical results.)

\begin{proposition}\label{pro:(1+2)DDegenerateBurgersEqCLs}
The quotient spaces of cosymmetries and of conservation-law characteristics
of the (1+2)-dimensional degenerate Burgers equation~\eqref{eq:(1+2)DDegenerateBurgersEq} 
by the corresponding subspaces of trivial objects%
\footnote{\label{fnt:DefOfTrivialCosymsAndCLChars}%
Similarly to generalized symmetries (cf.\ footnote~\ref{fnt:DefOfTrivialGenSyms}), 
a cosymmetry (resp.\ a conservation-law characteristic) of a system of differential equations~$\mathcal L$ is called \emph{trivial} 
if it vanishes on the solution set of~$\mathcal L$.
}
of the same kinds
are naturally isomorphic to the solution space of the (1+1)-dimensional backward heat equation $\gamma_t+\gamma_{yy}=0$,
$\gamma=\gamma(t,y)$.
The space of local conservation laws of the equation~\eqref{eq:(1+2)DDegenerateBurgersEq}
is naturally isomorphic to the space of conserved currents of the form
\begin{gather}\label{eq:(1+2)DDegenerateBurgersEqCCs}
(\gamma u,\, \tfrac12\gamma u^2,\,\gamma_yu-\gamma u_y)
\quad\mbox{with}\quad \gamma=\gamma(t,y)\colon\ \gamma_t+\gamma_{yy}=0.
\end{gather}
\end{proposition}

\begin{proof}
We follow the proof of Proposition~5 in~\cite{kont2017a}
and use the notation from the proof of Theorem~\ref{thm:(1+2)DDegenerateBurgersEqGenSyms}.
Cosymmetries of the equation~\eqref{eq:(1+2)DDegenerateBurgersEq} are differential functions of $u$,
$\gamma=\gamma[u]$, that satisfy the equation
\begin{gather}\label{eq:(1+2)DDegenerateBurgersEqConditionFoCosyms}
\mathrm D_0\gamma+u\mathrm D_1\gamma+\mathrm D_2^2\gamma=0
\end{gather}
on the manifold defined by the equation~\eqref{eq:(1+2)DDegenerateBurgersEq} and its differential consequences
in the corresponding jet space $\mathrm J^{\infty}(\mathbb R^3_z\times\mathbb R_u)$.
Since the equation~\eqref{eq:(1+2)DDegenerateBurgersEq} is of evolution type,
the cosymmetry~$\gamma$ can be assumed, up to equivalence of cosymmetries,
not to depend on derivatives of~$u$ involving differentiation with respect to~$z_0$.

Suppose that $\ord\gamma>-\infty$.
After expanding the equation~\eqref{eq:(1+2)DDegenerateBurgersEqConditionFoCosyms} and
substituting the expressions implied by the equation~\eqref{eq:(1+2)DDegenerateBurgersEq}
and its differential consequences for derivatives $u_\alpha$, where $\alpha_0=1$,
we collect the terms with derivatives of~$u$ of the highest order $r:=\ord\gamma+2$ 
appearing in the equation~\eqref{eq:(1+2)DDegenerateBurgersEqConditionFoCosyms},
which gives $2\sum_{k+l=r}\gamma_{u_{kl}}u_{k,l+2}=0.$
Splitting this equality with respect to $(r+2)$th order derivatives of~$u$
implies that $\gamma_{u_{kl}}=0$ for any~$(k,l)$ with $k+l=\ord\gamma$,
which contradicts the definition of order of a differential function.

Therefore, $\ord\gamma=-\infty$.
Separating terms with the first and zeroth degrees of~$u$
in the equations~\eqref{eq:(1+2)DDegenerateBurgersEqConditionFoCosyms},
we derive $\gamma_x=0$ and $\gamma_t+\gamma_{yy}=0$.
Any function $\gamma=\gamma(t,y)$ satisfying the last equation
is a characteristic of a conservation law of the Burgers equation~\eqref{eq:(1+2)DDegenerateBurgersEq},
which is associated with the conserved current~\eqref{eq:(1+2)DDegenerateBurgersEqCCs}.
\end{proof}

We consider cosymmetries and local conservation laws of various reduced equations
related to the Burgers equation~\eqref{eq:(1+2)DDegenerateBurgersEq} 
and interpret them as hidden cosymmetries and hidden conservation laws of~\eqref{eq:(1+2)DDegenerateBurgersEq}.

The reduced equations presented in Table~\ref{tab:LieReductionsOfCodim1},
except reduced equation~1.11, are second-order (1+1)-dimensional quasilinear evolution equations.
Reduced cosymmetries of such equations are of order not greater than zero,
and all these cosymmetries are conservation-law characteristics of the corresponding equations.
In fact, reduced equations~1.1--1.10 only admit reduced cosymmetries of order~$-\infty$.
The corresponding spaces of reduced cosymmetries are
\begin{gather*}\textstyle
1.1.\ \big\langle z_2M(2\kappa,\tfrac32,\tfrac\varepsilon4z_2^2),\,z_2U(2\kappa,\tfrac32,\tfrac\varepsilon4z_2^2)\big\rangle;\\[1ex] \textstyle
1.2.\ \big\langle z_2,\, z_2\int {\rm e}^{\varepsilon z_2^2/4}{\rm d}z_2-2\varepsilon {\rm e}^{\varepsilon z_2^2/4}\big\rangle;\\[1ex] \textstyle
1.3.\ \big\langle (z_2^2+2\varepsilon){\rm e}^{\varepsilon z_2^2/4},\, (z_2^2+2\varepsilon)\int {\rm e}^{-\varepsilon z_2^2/4}{\rm d}z_2+2\varepsilon z_2\big\rangle;\\[1ex]
1.4.\ \big\langle {\rm e}^{2z_1+\nu_1z_2},\,{\rm e}^{2z_1+\nu_2z_2}\big\rangle,\ \nu_{1,2}:=\tfrac12\beta\pm\tfrac12\sqrt{\beta^2+8\varepsilon'},
       \mbox{ \ if \ }\varepsilon'=1\mbox{ \ or \ }\varepsilon'=-1\mbox{ \ and \ }\beta^2>8,\\ \phantom{1.4.\ }
      \big\langle {\rm e}^{2z_1+\beta z_2/2},\,{\rm e}^{2z_1+\beta z_2/2}z_2\big\rangle\mbox{ \ if \ }\varepsilon'=-1\mbox{ \ and \ }\beta^2=8, \\ \phantom{1.4.\ }
      \big\langle {\rm e}^{2z_1+\beta z_2/2}\cos\mu z_2,\,{\rm e}^{2z_1+\beta z_2/2}\sin\mu z_2\big\rangle,\ \mu:=\tfrac12\sqrt{8-\beta^2},\mbox{ \ if \ }\varepsilon'=-1\mbox{ \ and \ }\beta^2<8;\\[1ex]
1.5.\ \{0\}\mbox{ \ if \ }\delta=0\quad\mbox{and}\quad \big\langle {\rm e}^{2z_2-4z_1}\big\rangle\mbox{ \ if \ }\delta=1;\\[1ex]
1.6.\ \big\langle 1,\,z_2\big\rangle\mbox{ \ if \ }\delta'=0\quad\mbox{and}\quad \big\langle 1,\,{\rm e}^{z_2}\big\rangle\mbox{ \ if \ }\delta'=1;\\[1ex]
1.7,\ 1.10.\ \big\langle 1\big\rangle;\\[1ex]
1.8,\ 1.9.\ \big\{\gamma(z_1,z_2)\mid\gamma_1+\gamma_{22}=0\big\}.
\end{gather*}
Here $M(a,b,\omega)$ are $U(a,b,\omega)$ are the Kummer and Tricomi functions,
or the confluent hypergeometric functions of the first and the second kinds,
respectively, which are solutions of Kummer's (confluent hypergeometric) equation
$\omega\varphi_{\omega\omega}+(b-\omega)\varphi_\omega-a\varphi=0$.
Therefore, all cosymmetries of the above reduced equations are induced by cosymmetries of the equation~\eqref{eq:(1+2)DDegenerateBurgersEq},
and hence similar claims hold for local conservation laws and their characteristics.
This cosymmetry induction can be explained using the adjoint variational principle
and its interpretation in~\cite{ibra2007a} within the framework of symmetry analysis.
See also the discussion in~\cite{popo2008a}.
Consider the equation~\eqref{eq:(1+2)DDegenerateBurgersEq} together with the adjoint equation
\begin{gather}\label{eq:(1+2)DDegenerateBurgersAdjointEq}
v_t+uv_x+v_{yy}=0.
\end{gather}
According to \cite[Theorem 3.4]{ibra2007a}, any symmetry of the equation~\eqref{eq:(1+2)DDegenerateBurgersEq}
can be extended to a symmetry of the system~\eqref{eq:(1+2)DDegenerateBurgersEq}, \eqref{eq:(1+2)DDegenerateBurgersAdjointEq},
and the extension can be realized in such a way that
the extended symmetry is also a symmetry of the Lagrangian $v(u_t+uu_x-u_{yy})$ of this system.
Thus, the extension of the vector fields~$D^t$, $D^x$, $P^t$, $G^x$, $P^y$ and $P^x$ to~$v$
gives $D^t-v\p_v$, $D^x+v\p_v$, $P^t$, $G^x$, $P^y$ and $P^x$, respectively.
Additionally to~\eqref{eq:(1+2)DDegenerateBurgersAdjointEq},
cosymmetries of~\eqref{eq:(1+2)DDegenerateBurgersEq} satisfy the constraint $v_x=0$.
The above cosymmetries of reduced equations are solutions of reduced systems for the system $v_t+v_{yy}=0$, $v_x=0$
with respect to the counterparts of one-dimensional $G$-inequivalent subalgebras of~$\mathfrak g$
listed in Section~\ref{sec:BurgersSystemLiereductionsOfCodim1}.
More specifically, to reduce cosymmetries in a way consistent with reductions of~\eqref{eq:(1+2)DDegenerateBurgersEq},
we modify one-dimensional subalgebras~$\mathfrak g_{1.1}$--$\mathfrak g_{1.10}$
by  projecting their basis elements to the space with the coordinates~$(t,x,y)$ and
by extending the obtained vector fields to~$v$
via combining with~$v\p_v$, which is an obvious Lie-symmetry vector field
of the system~\eqref{eq:(1+2)DDegenerateBurgersEq},~\eqref{eq:(1+2)DDegenerateBurgersAdjointEq}.
The algorithm of finding the proper coefficients~$\alpha$ of~$v\p_v$ is the following.
For each $i\in\{1,\dots,10\}$, we complete $(z_1,z_2,w)$ with an auxiliary variable~$z_0$
such that in the new coordinates $(z_0,z_1,z_2,w)$ the basis vector field~$Q$ of the algebra~$\mathfrak g_{1.i}$
is proportional to the vector field~$\p_{z_0}$.
A consistent ansatz for cosymmetries is $\gamma=\lambda^{-1}\mathrm J\tilde\gamma(z_1,z_2)$,
where $\mathrm J$ denotes the Jacobian of $(z_0,z_1,z_2)$ with respect to $(t,x,y)$,
and $\lambda$ is the multiplier that is canceled in the course of deriving reduced equation~$1.i$
after substituting ansatz~$1.i$ for~$u$ into the equation~\eqref{eq:(1+2)DDegenerateBurgersEq}.
This ansatz should correspond to the one-dimensional algebra spanned by $Q+\alpha v\p_v$.
We choose $z_0$ equal to $\ln|t|$, $\ln|t|$, $\ln|t|$,  $t$, $\ln|x|$, $t$, $y$, $x/t$, $x$ and~$y$,
and thus the coefficient~$\alpha$ takes the values
$-4\kappa+1$, $1$, $-3$, $-2$, $-2$, $0$, $0$, $0$, $0$, $0$
for Cases~1.1, \dots, 1.10, respectively.

Therefore, genuine hidden cosymmetries
and hidden conservation laws of~\eqref{eq:(1+2)DDegenerateBurgersEq}
are associated only with reduced equation~1.11, which is the transport equation
\[
w_1+ww_2=0.
\]
We denote
\[
\zeta^k:=\left(\frac1{w_2}\bar{\mathrm D}_{z_2}\right)^k(z_2-wz_1),\quad k\in\mathbb N_0:=\mathbb N\cup\{0\},
\]
where $\bar{\mathrm D}_{z_2}:=\p_2+\sum_{k=0}^\infty\omega^{(k+1)}\p_{\omega^{(k)}}$
is the reduced operator of total derivative with respect to~$z_2$, $\omega^{(k)}:=\p_2^{\;k}w$, $k\in\mathbb N_0$.
Thus, $\zeta^0:=z_2-wz_1$, $\zeta^1:=w_2^{\,-1}-z_1$, $\zeta^i:=(w_2^{\,-1}\bar{\mathrm D}_{z_2})^{i-1}w_2^{\,-1}$, \mbox{$i=2,3,\dots$}.
A~symbol with~$[\zeta]$, like $f[\zeta]$, denotes a differential function of~$w$
that depends at most on~$w$ and a finite number of~$\zeta^k$,
$f[\zeta]=f(w,\zeta^0,\dots,\zeta^l)$ for some $l\in\mathbb N_0$.
The quotient spaces of generalized symmetries, cosymmetries and conservation-law characteristics
of reduced equation~1.11 by the corresponding subspaces of trivial objects of the same kinds
are naturally isomorphic to the spaces
\begin{gather*}
\big\{w_2\eta[\zeta]\p_w\big\},\quad
\big\{\gamma[\zeta]\big\},\quad
\left\{\sum_{k=0}^\infty\left(-\frac1{w_2}\bar{\mathrm D}_{z_2}\right)^k\p_{\zeta^k}\rho[\zeta]\right\},\quad
\end{gather*}
respectively.
Here $\eta$, $\gamma$ and~$\rho$ run through the space of the above differential functions. 
See the appendix in~\cite{baik1989a} for computing the generalized symmetries of the transport equation.
The space of local conservation laws of reduced equation~1.11 is naturally isomorphic
to the quotient of the space of conserved currents of the form
\[
\big\{\big(w_2\rho[\zeta],\,ww_2\rho[\zeta]\big)\big\}
\]
by its subspace singled out by the constraint \[\sum_{k=0}^\infty(-w_2^{\,-1}\bar{\mathrm D}_{z_2})^k\p_{\zeta^k}\rho[\zeta]=0.\]
In view of Theorem~\ref{thm:(1+2)DDegenerateBurgersEqGenSyms}
and results from Section~\ref{sec:BurgersSystemLiereductionsOfCodim1},
generalized symmetries of~\eqref{eq:(1+2)DDegenerateBurgersEq}
induce only generalized symmetries of reduced equation~1.11
that are equivalent to Lie symmetries of this equation,
and the induced Lie symmetries constitute a finite-dimensional subalgebra
of the entire infinite-dimensional maximal Lie invariance algebra of reduced equation~1.11.
Up to the corresponding equivalences,
the cosymmetries and the conservation-law characteristics of reduced equation~1.11
that are induced by those of the equation~\eqref{eq:(1+2)DDegenerateBurgersEq}
are exhausted by the constant ones.
The only linearly independent conservation law induced by a conservation law of~\eqref{eq:(1+2)DDegenerateBurgersEq}
contains the conserved current $(w,w^2/2)$.
All the other generalized symmetries, cosymmetries and local conservation laws of reduced equation~1.11
are respectively nontrivial hidden generalized symmetries, cosymmetries and local conservation laws
of the (1+2)-dimensional degenerate Burgers equation~\eqref{eq:(1+2)DDegenerateBurgersEq}.

\begin{remark}
Reduced equations~1.1--1.9 have no genuine potential conservation laws. 
For reduced equations~1.8 and~1.9, which coincide with the (1+1)-dimensional linear heat equation,
this claim is nontrivial and was proved in~\cite[Theorem~5]{popo2008a}.
For reduced equations~1.6$_{\delta=0}$ and~1.7, the claim follows from results of~\cite{popo2005b} and~\cite{poch2017a},
respectively, but it can be checked for all reduced equations~1.1--1.7 in a similar way.
At the same time,
reduced equation~1.10, which is the Burgers equation,
possesses potential conservation laws with conserved currents
\mbox{${\rm e}^{-v/2}\big(\gamma,\,\gamma_2+\frac12\gamma w\big)$}, 
where the parameter function $\gamma=\gamma(z_1,z_2)$ runs through the solution set
of the (1+1)-dimensional backward heat equation $\gamma_1+\gamma_{22}=0$,
and the potential~$v=v(z_1,z_2)$ is defined by the system $v_1=w_2-\frac12w^2$, $v_2=w$, see~\cite{popo2005b}.
Reduced equation~1.11 also admits genuine potential conservation laws.
The above potential conservation laws of reduced equations~1.10 and~1.11 are nontrivial hidden potential conservation laws
of the (1+2)-dimensional degenerate Burgers equation~\eqref{eq:(1+2)DDegenerateBurgersEq}.
\end{remark}

\section{Common solutions of degenerate and nondegenerate Burgers equations in (1+2) dimensions}
\label{sec:CommonSolutionsOf(1+2)DDegenerateAndNondegenerateBurgersEqs}

Consider the solutions of
the (1+2)-dimensional degenerate Burgers equation~\eqref{eq:(1+2)DDegenerateBurgersEq}
that satisfy the additional constraint $u_{xx}=0$.
Such solutions obviously exhaust the set of common solutions of~\eqref{eq:(1+2)DDegenerateBurgersEq}
and the (1+2)-dimensional Burgers equation
\begin{gather}\label{eq:(1+2)DBurgersEq}
u_t+uu_x-u_{xx}-u_{yy}=0.
\end{gather}
To the best of our knowledge,
the latter equation was considered for the first time in~\cite{edwa1995a},
which included the computation of its maximal Lie invariance algebra
and a preliminary analysis of its two-step Lie reductions to ordinary differential equations.
As a physical model, this equation was first derived in~\cite{raja2008a}
for describing the wave phase of two-dimensional simple sound waves in weakly dissipative flows.
Therein, Lie reductions of the equation~\eqref{eq:(1+2)DBurgersEq}
were comprehensively studied in an optimized way
and wide families of new exact solutions for this equation were found.
As an equation with exceptional symmetries among (1+2)-dimensional diffusion--convection equations,
the equation~\eqref{eq:(1+2)DBurgersEq} was singled out
in the course of the group classification of such equations in~\cite{deme2008a},
where some exact solutions of~\eqref{eq:(1+2)DBurgersEq} were also constructed.

The integration of the constraint $u_{xx}=0$ gives the ansatz
\begin{gather}\label{eq:(1+2)DDegenerateBurgersEqAnsatzLinearInX}
u=w^1(t,y)x+w^0(t,y)
\end{gather}
for~$u$, which (separately) reduces both the equations~\eqref{eq:(1+2)DDegenerateBurgersEq} and~\eqref{eq:(1+2)DBurgersEq}
to a system of two (1+1)-dimensional evolution equations with respect to $(w^0,w^1)$,
\begin{gather}
w^1_t-w^1_{yy}+(w^1)^2=0,
\label{eq:(1+2)DDegenerateBurgersEqSystemForSolutionsAffineInX1}\\
w^0_t-w^0_{yy}+w^1w^0=0.
\label{eq:(1+2)DDegenerateBurgersEqSystemForSolutionsAffineInX2}
\end{gather}
This means that the generalized vector field $u_{xx}\p_u$
is a generalized conditional symmetry of both the equations~\eqref{eq:(1+2)DDegenerateBurgersEq} and~\eqref{eq:(1+2)DBurgersEq}.
The ansatz~\eqref{eq:(1+2)DDegenerateBurgersEqAnsatzLinearInX}
generalizes ansatzes 1.5 with $\delta=0$, 1.8 and~1.9,
which are derived from~\eqref{eq:(1+2)DDegenerateBurgersEqAnsatzLinearInX}
by setting $w^0=0$, $w^1=1/t$ and $w^1=0$, respectively.

We review the other families of exact solution associated with the constraint $u_{xx}=0$
that were constructed in~\cite[Section~14]{kont2017a}; see also~\cite{raja2008a}.

The equation~\eqref{eq:(1+2)DDegenerateBurgersEqSystemForSolutionsAffineInX1}
is the nonlinear heat equation with purely quadratic nonlinearity.
Looking for its stationary solutions,
we integrate the system $w^1_t=0$, $w^1_{yy}=(w^1)^2$
whose general solution is
$w^1=\wp\big(y/\sqrt6+C_2;0,C_1\big)$.
Here $C_1$ and~$C_2$ are arbitrary constants
and  $\wp=\wp(z;\textsl{g}_2,\textsl{g}_3)$ is the Weierstrass elliptic function,
which satisfies the differential equation
\begin{gather}\label{eq:WeierstrassEq}
(\wp_z)^2 = 4\wp^3-\textsl{g}_2\wp-\textsl{g}_3,
\end{gather}
where the numbers $\textsl{g}_2$ and $\textsl{g}_3$ called invariants
are respectively equal to $0$ and~$C_1$ for~$w^1$.
The constant~$C_2$ can be set to zero up to shifts of~$y$,
which are Lie symmetry transformations of both the initial equation~\eqref{eq:(1+2)DDegenerateBurgersEq}
and the reduced system~\eqref{eq:(1+2)DDegenerateBurgersEqSystemForSolutionsAffineInX1}--\eqref{eq:(1+2)DDegenerateBurgersEqSystemForSolutionsAffineInX2}.
For such general values of~$w^1$, the essential invariance algebra of
the equation~\eqref{eq:(1+2)DDegenerateBurgersEqSystemForSolutionsAffineInX2} is $\langle\p_t,w^0\p_{w^0}\rangle$.
Using the subalgebra $\langle\p_t+C_3w^0\p_{w^0}\rangle$ of this algebra with an arbitrary constant~$C_3$,
we construct the ansatz $w^0={\rm e}^{C_3t}\varphi(z)$ with $z=y/\sqrt6$,
which reduces~\eqref{eq:(1+2)DDegenerateBurgersEqSystemForSolutionsAffineInX2} to the Lam\'e equation
\begin{gather}\label{eq:(1+2)DDegenerateBurgersEqLameEq}
\varphi_{zz}=6\big(C_3+\wp(z;0,C_1)\big)\varphi.
\end{gather}
As a result, we get the following solution family
of the (1+2)-dimensional Burgers equations~\eqref{eq:(1+2)DDegenerateBurgersEq} and~\eqref{eq:(1+2)DBurgersEq}:
\begin{gather}\label{eq:(1+2)DDegenerateBurgersEqSolutionsWithLameFunctions}
u=\wp\big(y/\sqrt6;0,C_1\big)x+{\rm e}^{C_3t}\varphi\big(y/\sqrt6\big),
\end{gather}
where $C_1$ and~$C_3$ are arbitrary constants
and $\varphi$ is the general solution of Lam\'e equation~\eqref{eq:(1+2)DDegenerateBurgersEqLameEq}.

In the degenerate case $C_1=0$, the above value of~$w^1$ coincides, up to shifts in~$y$, with the function $w^1=6y^{-2}$.
Then the essential invariance algebra of the equation~\eqref{eq:(1+2)DDegenerateBurgersEqSystemForSolutionsAffineInX2}
becomes wider~\cite[Section~9.9]{ovsi1982}
and, which is more important, all solutions of this equation
can be expressed in terms of solutions of the linear heat equation
using the Darboux transformation,
\[
w^0=\mathrm{DT}[y,y^3+6t]\theta=\theta_{yy}-\frac3y\theta_y+\frac3{y^2}\theta.
\]
Here $\theta=\theta(t,y)$ is an arbitrary solution of the linear heat equation $\theta_t=\theta_{yy}$,
and, given solutions~$f^1$, \dots, $f^k$ of this equation with $\mathrm W(f^1,\dots,f^k)\ne0$,
the action of the corresponding Darboux operator on a solution~$f$ is defined~by 
\[
  \mathrm{DT}[f^1,\dots,f^k]f=\frac{\mathrm W(f^1,\dots,f^k,f)}{\mathrm W(f^1,\dots,f^k)},
\]
where $\mathrm W$'s denote the Wronskians of the indicated tuples of functions with respect to the `space' variable~$y$~\cite{popo2008a}.
We construct the following solution family of the (1+2)-dimensional Burgers equations~\eqref{eq:(1+2)DDegenerateBurgersEq} and~\eqref{eq:(1+2)DBurgersEq}:
\begin{gather}\label{eq:(1+2)DDegenerateBurgersEqSolutionSetIn1SolutionOfLHEq}
u=6\frac x{y^2}+\theta_{yy}-\frac3y\theta_y+\frac3{y^2}\theta,
\end{gather}
where $\theta=\theta(t,y)$ is an arbitrary solution of the linear heat equation $\theta_t=\theta_{yy}$.

Finally, setting $w^1$ to be the similarity solution of the nonlinear heat equation with purely quadratic nonlinearity
that is constructed in~\cite{bara2002d} (see also \cite[Section~5.1.1.1.1]{Polyanin&Zaitsev2012})
and finding similarity solutions of the corresponding equation~\eqref{eq:(1+2)DDegenerateBurgersEqSystemForSolutionsAffineInX2},
we obtain the solution
\begin{gather}\label{eq:(1+2)DDegenerateBurgersEqSolutionInHeunFunctions}
\begin{split}
u={}&12(4\pm\sqrt6)\frac{y^2+(18\pm8\sqrt6)t}{(y^2+10\lambda_\pm t)^2}x
+|t|^{\nu+3/2}\exp\left(-\frac{y^2}{4t}\right)\frac{C_1y+C_2|t|^{1/2}}{(y^2+10\lambda_\pm t)^2}
\\[.5ex] &\times
\mathop{\rm HeunC}\left(\frac52\lambda_\pm,-\frac12,-5,\frac58\lambda_\pm(4\nu+1),-\frac52\lambda_\pm\nu-\frac{59}8\mp\frac{29}8\sqrt6,\frac{-y^2}{10\lambda_\pm t}\right)
\end{split}
\end{gather}
of the equations~\eqref{eq:(1+2)DDegenerateBurgersEq} and~\eqref{eq:(1+2)DBurgersEq}.
Here $\lambda_\pm=3\pm\sqrt6$,
$\mathop{\rm HeunC}(\alpha,\beta,\gamma,\delta,\eta,z)$
is the confluent Heun function,
which is the solution of the following Cauchy problem for the confluent Heun equation with respect to $Y=Y(z)$:
\begin{gather*}
z(z-1)Y_{zz}+\big(\alpha z(z-1)+(\beta+1)(z-1)+(\gamma+1)z\big)Y_z\\
\qquad{}+\frac12\big(\alpha(\beta+1)(z-1)+\alpha(\gamma+1)z+2\delta z+(\beta+1)(\gamma+1)+2\eta-1\big)Y=0,\\[.5ex]
Y(0)=1,\quad Y_z(0)=\frac12\left(\frac{2\eta-1}{\beta+1}+\gamma+1-\alpha\right).
\end{gather*}
The solution family~\eqref{eq:(1+2)DDegenerateBurgersEqSolutionInHeunFunctions}
can be additionally extended by some point symmetry transformations of the respective equation. 
Thus, for the equation~\eqref{eq:(1+2)DDegenerateBurgersEq}, 
these are the transformations from Theorem~\ref{thm:BurgersSystemCompletePointSymGroup} with $\delta_1=1$.


\appendix

\section{Exact solutions of the heat equation}\label{sec:ExactSolutionsOfHeatEq}

The maximal Lie invariance algebra of (1+1)-dimensional linear heat equation $w_t=w_{yy}$,
which is denoted $\mathfrak a_8$ or $\mathfrak a_9$
in Section~\ref{sec:BurgersSystemLiereductionsOfCodim1}, is
\begin{gather*}
\begin{split}
\mathfrak a={}\big\langle&P^t=\p_t,\ P^x=\p_y,\ G^y=2t\p_y-yw\p_w,\ D=2t\p_t+y\p_y,
\\&
K=4t^2\p_t+4ty\p_y-(y^2+2t)w\p_w,\ I=w\p_w,\ f(t,y)\p_w\big\rangle,
\end{split}
\end{gather*}
where the function $f=f(t,y)$ runs through the solution set of the linear heat equation $f_t=f_{yy}$.
The complete point symmetry group of the heat equation
is generated by the one-parameter groups associated with vector fields from the algebra~$\mathfrak a$
and two discrete transformations:
alternating the sign of~$y$, $(t,y,w)\mapsto(t,-y,w)$,
and alternating the sign of~$w$, $(t,y,w)\mapsto(t,y,-w)$.

Lie invariant solutions of the heat equation
were comprehensively studied in~\cite[Example 3.17]{olve1993b}.
More precisely, all solutions that are invariant with respect to
inequivalent one-dimensional subalgebras of~$\mathfrak a$ were constructed in explicit form.
For convenience of referring, we list these solutions with the corresponding subalgebras below,
simultaneously correcting minor cumulative misprints.
A large set of solutions in closed form for the heat equation was collected in~\cite{widd1975a}.
See also~\cite{ivan2008a}.
Hereafter $C_1$ and $C_2$ are arbitrary constants, $a$ and $b$ are constant parameters.
\begin{gather}\label{sol_cylindricUV}
w=t^a\exp\left({-\frac{y^2}{8t}}\right)\left(C_1\mathop{\rm U}\Big(2a+\frac12,\frac y{\sqrt{2t}}\,\Big)+C_2\mathop{\rm V}\Big(2a+\frac12,\frac y{\sqrt{2t}}\,\Big)\right),
\quad \langle D+aI\rangle,
\end{gather}
where $\mathop{\rm U}(b,z)$ and $\mathop{\rm V}(b,z)$ are parabolic cylinder functions~\cite[$\S$ 19.1]{Abramowitz&Stegun1970}
constituting the standard fundamental set of solutions of the equation $\varphi''=\big(\frac14z^2+b\big)\varphi$;
\begin{gather}\label{sol_cylindricW}
\begin{split}
w={}&{(4t^2+1)^{-1/4}}{\exp\left({-\dfrac{ty^2}{4t^2+1}-\dfrac a2\arctan(2t)}\right)}\\
&\times\left(C_1\mathop{\rm W}\left(-\dfrac a2,\dfrac{\sqrt2\,y}{\sqrt{4t^2+1}}\right)
   +C_2\mathop{\rm W}\left(-\dfrac a2,-\dfrac{\sqrt2\,y}{\sqrt{4t^2+1}}\right)\right),
\quad \langle K+P^t+aI\rangle,
\end{split}
\end{gather}
where the parabolic cylinder functions $\mathop{\rm W}(b,\pm z)$ are the standard linearly independent solutions 
of the equation $\varphi''=-\big(\frac14z^2-b\big)\varphi$ \cite[$\S$\,19.17]{Abramowitz&Stegun1970};
\begin{gather}\label{sol_Airy}
w=\exp\left(yt+\frac23t^3\right)\left(C_1\mathop{\rm Ai}(y+t^2)+C_2\mathop{\rm Bi}(y+t^2)\right),
\quad \langle P^t+G^y+aI\rangle,
\end{gather}
where the Airy wave functions $\mathop{\rm Ai}$ and $\mathop{\rm Bi}$ are the standard linearly independent solutions of the Airy equation $\varphi''=z\varphi$;
\begin{gather*}
w=C_1{\rm e}^{t+y}+C_2{\rm e}^{t-y},\quad w=C_1x+C_2,\quad w={\rm e}^{- t}(C_1\cos y+C_2\sin y),
\quad \langle P^t+aI\rangle,
\end{gather*}
with $a=1,0,-1$, respectively.
The solutions~\eqref{sol_cylindricUV}, \eqref{sol_cylindricW} and~\eqref{sol_Airy}
can be expressed via hypergeometric functions or, for certain values of parameters, via Bessel functions.

From~\eqref{sol_cylindricUV} with $a=0$, we get the particular solution in terms of the error function
\begin{gather}\label{sol_errf}
w=\mathop{\rm erf}\left(\frac y{2\sqrt t}\right), \quad \mathop{\rm erf}(z):=\frac2{\sqrt\pi}\int^z_0 {\rm e}^{-\xi^2}\,{\rm d}\xi.
\end{gather}
It is obvious that there is also the particular solution in terms of the complementary error function~$\mathop{\rm erfc}$ of the same argument.
The solutions~\eqref{sol_cylindricW} and~\eqref{sol_errf} were presented in~\cite{olve1993b} with minor misprints.

Setting $a=-(n+1)/2$ in~\eqref{sol_cylindricUV} for $C_2=0$ and
using the formula \smash{$\mathop{\rm U}\left(-n-\tfrac12,z\right)={\rm e}^{-\frac{z^2}4}{\rm He}_n(z)$},
up to a constant multiplier, we get the particular solution
\begin{gather*}
w=t^{-\frac{n+1}2}{\rm e}^{-\frac{y^2}{4t}}\,{\rm He}_n\left(\frac y{\sqrt{2t}}\right),
\quad\mbox{or}\quad w=t^{-\frac{n+1}2}{\rm e}^{-\frac{y^2}{4t}}\,{\rm H}_n\left(\frac y{2\sqrt t}\right),
\end{gather*}
in terms of $n$th-order Hermite polynomials,
which are related as $\mathop{\rm He}_n(z)=2^{-\frac n2}\mathop{\rm H}_n\left(z/\sqrt2\right)$.
For $n=0$ it becomes the source solution
\begin{gather*}
\smash{w=\frac1{\sqrt t}{\rm e}^{-\frac{y^2}{4t}}}.
\end{gather*}
The relation $\mathop{\rm V}\left(n+\frac12,z\right)=\sqrt{\frac2\pi}{\rm e}^\frac{z^2}4\mathop{\rm He}_n^*(z)$,
where $\mathop{\rm He}_n^*(z)=(-{\rm i})^n\mathop{\rm He}_n({\rm i}z)$, leads to the heat polynomials
$y$, $y^2+2t$, $y^3+6ty$, $y^4+12ty^2+12t^2$, \dots, which can be written uniformly as
\begin{gather*}
w=n!\sum_{k=0}^{[n/2]}\frac{t^k}{k!}\frac{y^{n-2k}}{(n-2k)!},\quad n\in\mathbb N,
\end{gather*}
where $[n/2]$ is the floor function of $n/2$, i.e., the greatest integer less than or equal to $n/2$.

More solutions of the heat equation can be generated from the above solutions using the action by point symmetry transformations,
\[
\tilde w=\frac{\delta_3}{\sqrt{1+4\delta_6t}}
\exp\left({-\frac{\delta_6y^2+\delta_5y-\delta_5^2t}{1+4\delta_6t}}\right)
w\left(\frac{\delta_4^2t}{1+4\delta_6t}-\delta_2,\,
\delta_4\frac{y-2\delta_5t}{1+4\delta_6t}-\delta_1\right)
+h(t,y),
\]
where
$\delta_1,\ldots,\delta_6$ are arbitrary constants with $\delta_3\delta_4\ne0$ and
$w$ and $h$ are arbitrary solutions of the linear heat equation \cite[Example~2.41]{olve1993b}.
For generating solutions, we also can use the action of differential operators associated
with Lie symmetry vector fields of the heat equation.
In fact, it suffices to use only two of these operators, $\mathrm D_x$ and $2t\mathrm D_x+x$ \cite[Example~5.21]{olve1993b},
where $\mathrm D_x$ denotes the operator of total derivative with respect to~$x$.
Note that such action preserves the set of invariant solutions.
Wide families of explicit solutions of the heat equation can be constructed by linearly combining
solutions from the presented families with various values of parameters.

\section{New exact solutions of some nonlinear diffusion--convection equations}\label{sec:ExactSolutionsOfNDCEs}

As a by-product, 
in Sections~\ref{sec:BurgersSystemLieReductionOfCodim2} and~\ref{sec:CommonSolutionsOf(1+2)DDegenerateAndNondegenerateBurgersEqs} 
we construct families of inequivalent exact invariant solutions of several distinguished equations
similar to reduced ones, which we present below.
These families can be extended with symmetry transformations of the corresponding equations.
To the best of our knowledge, most of the solutions are new.

In particular, for the nonlinear diffusion equation with power nonlinearity of degree~$-1/2$, \[v_t=(v^{-1/2}v_x)_x,\]
we find the solutions
(for certain domains of $(t,x)$)\footnote{%
This domain is defined in the following way.
If $v=w^2$, where $w$ is a solution of the equation $ww_t=w_{xx}$ (resp. the equation $ww_t=-w_{xx}$),
then~$v$ satisfies the equation $v_t=(v^{-1/2}v_x)_x$
for $(t,x)$ with $w(t,x)>0$ (resp. with $w(t,x)<0$).
}
\begin{gather}\label{eq:SolutionsOfPower12DiffusionEq}
\begin{split}
&v=\frac{36t^2}{(x^2+C_1t^{4/3})^2},
 \quad
 v=9t^{-2/3}\left(\frac
 {C_1\mathop{\rm Ai} (t^{-1/3}x)+C_2\mathop{\rm Bi} (t^{-1/3}x)}
 {C_1\mathop{\rm Ai}'(t^{-1/3}x)+C_2\mathop{\rm Bi}'(t^{-1/3}x)}
 \right)^2,
\\[.5ex]
&v=\frac{36t^2x^2}{(x^3+C_1t^3)^2},\quad
 v=\frac{4(C_1{\rm e}^{x/t}-C_2{\rm e}^{-x/t})^2}{\big(C_1{\rm e}^{x/t}(x-t)-C_2{\rm e}^{-x/t}(x+t)\big)^2},\\[.5ex]
&v=\frac{4\big(-C_1\sin(x/t)+C_2\cos(x/t)\big)^2}{\big(C_1(x\cos(x/t)-t\sin(x/t))+C_2(x\sin(x/t)+t\cos(x/t))\big)^2},\\[.5ex]
&v=\left(\wp\big(x/\sqrt6;0,C_1\big)t+\varphi\big(x/\sqrt6\big)\right)^2,\quad 
 v=(6x^{-2}t+C_4x^3)^2.
\end{split}
\end{gather}
Here the Airy wave functions $\mathop{\rm Ai}$ and $\mathop{\rm Bi}$ are the standard linearly independent solutions
of the Airy equation $\tilde\psi_{\tilde\omega\tilde\omega}=\tilde\omega\tilde\psi$,
$\wp=\wp(z;\textsl{g}_2,\textsl{g}_3)$ is the Weierstrass elliptic function satisfying the equation~\eqref{eq:WeierstrassEq},
and $\varphi$ is the general solution of Lam\'e equation~\eqref{eq:(1+2)DDegenerateBurgersEqLameEq}.
Only the first, the third and the two last solution families are known in the literature.
Extended with Lie symmetry transformations of the $v^{-1/2}$-diffusion equation,
the first and the third families coincide with the fourth and fifth families from 5.1.10.8.2$^\circ$ in~\cite{Polyanin&Zaitsev2012} under setting $m=-1/2$,
respectively,
and the third family therein with $m=-1/2$ is contained in the first family from~\eqref{eq:SolutionsOfPower12DiffusionEq}.
The penultimate family was constructed in~\cite{amer1990,king1992a}; see also \cite[Eq.~5.1.10.7]{Polyanin&Zaitsev2012}.
For $C_1=0$, it degenerates, up to shifts of~$t$, to the last family. 
The common member of the first, the third and the last families of solutions is $v=36t^2/x^4$.
In addition, the equation  $v_t=(v^{-1/2}v_x)_x$ admits simple shift-invariant solutions
$v=1$, $v=x^2$, $v=4/(x-t)^2$, $v=4\tanh(x-t)$, $v=4\coth(x-t)$, $v=4\tan(x+t)$.

Another equation is the nonlinear diffusion--convection equation \[v_t=(v^{-1/2}v_x)_x+v^{-1/2}v_x,\]
for which we construct the new explicit solutions
\begin{gather*}
u=\frac{t^2}{(x+2\ln|t|)^2},\quad
u=-\frac14{\rm e}^{-x}\left(\frac{C_1J_1(t^{-1}{\rm e}^{-x/2})+C_2Y_1(t^{-1}{\rm e}^{-x/2})}{C_1J_0(t^{-1}{\rm e}^{-x/2})+C_2Y_0(t^{-1}{\rm e}^{-x/2})}\right)^2,\\[.5ex]
u= \frac14{\rm e}^{-x}\left(\frac{C_1I_1(t^{-1}{\rm e}^{-x/2})-C_2K_1(t^{-1}{\rm e}^{-x/2})}{C_1I_0(t^{-1}{\rm e}^{-x/2})+C_2K_0(t^{-1}{\rm e}^{-x/2})}\right)^2.
\end{gather*}

\section*{Acknowledgments}

The authors thank Dmytro Popovych and Galyna Popovych for helpful discussions and interesting comments. 
The authors are also grateful to the reviewer for carefully reading the paper and providing suggestions for improving it.
ROP and CS would like to express their gratitude for the reciprocal hospitality shown by their two institutions, 
OV is grateful to the University of Cyprus for the hospitality and support.
The research of ROP was supported by the Austrian Science Fund (FWF), projects P28770 and~P29177.
OV acknowledges the financial support of her research
within the L'Or\'eal--UNESCO For Women in Science International Rising Talents Programme 
and support of NAS of Ukraine within the project 0121U110543.

\end{document}